[1994/12/01]

\documentclass{amsart}
\usepackage{amssymb}
\usepackage{amsmath,amssymb,amsfonts,amsthm,graphics,
latexsym, amscd, amsfonts, epsfig, eepic,epic}
\usepackage{mathrsfs}
\usepackage{algorithm}
\usepackage{algorithmic}
\usepackage{color}

\usepackage{eucal}
\usepackage{eufrak}
\usepackage[all]{xypic}
\usepackage{xspace}

\newtheorem{thm}{Theorem}[section]
\newtheorem{lem}[thm]{Lemma}
\newtheorem{prop}[thm]{Proposition}
\newtheorem{cor}[thm]{Corollary}

\theoremstyle{definition}
\newtheorem{defn}{Definition}[section]

\theoremstyle{remark}
\newtheorem{rem}{Remark}[section]

\makeatletter \@addtoreset{equation}{section}

\newcommand{\lemref}[1]{Lemma~\ref{#1}}

\def\a{\alpha}
\def\b{\beta}

\def\l{\lambda}

\def\p{\partial}
\def\vphi{\varphi}

\def\L{\Lambda}

\def\tr{\rm tr}

\def\Int{{\rm Int}\,}

\def\and{\quad{\rm and}\quad}

\def\eps{\epsilon}

\let\lra=\longrightarrow

\def\mapright{\xrightarrow}
\def\mapright\#1{\,\smash{\mathop{\lra}\limits^{\#1}}\,}

\def\om{\omega}
\def\Om{\Omega}

\def\tri{\triangle}

\numberwithin{equation}{section}

\DeclareMathOperator{\Vol}{Vol}
\DeclareMathOperator{\Ric}{Ric}

\DeclareMathOperator{\osc}{Osc}

\makeatletter

\newcommand{\Rmnum}[1]{\expandafter\@slowromancap\romannumeral #1@}

\title[Geodesics in the space of K\"ahler cone metrics]{Geodesics in the space of K\"ahler cone metrics}

\author [Simone Calamai]{Simone Calamai}
  \address{Scuola Normale Superiore di Pisa - Piazza dei Cavalieri, 7 - 56126 Pisa, Italy }
  \email{simocala@gmail.com}
\author [Kai Zheng]{Kai Zheng}
  \address{Institut Fourier, Universit\'{e} Joseph Fourier (Grenoble I), UMR 5582 CNRS-UJF, BP 74, 38402 Saint-Martin-d'H\`{e}res, France}
  \email{zheng@math.uni-hannover.de}
\thanks{The first named author is partially supported by INdAM grants.
The second named author is partially supported by ANR project
"Flots et Op\'erateurs G\'eom\'etriques" (ANR-07-BLAN-0251-01).}
\begin{document}
\maketitle
\begin{abstract}
In this paper, we study the Dirichlet problem of the geodesic equation in the space of K\"ahler cone metrics $\mathcal H_\b$; that is equivalent to a homogeneous complex Monge-Amp\`ere equation whose boundary values consist of K\"ahler metrics with cone singularities. Our approach concerns the generalization of the space defined in Donaldson \cite{MR2975584} to the case of K\"ahler manifolds with boundary; moreover we introduce a subspace $\mathcal H_C$ of $\mathcal H_\b$ which we define by prescribing appropriate geometric conditions. Our main result is the existence, uniqueness and regularity of $C^{1,1}_\b$ geodesics whose boundary values lie in $\mathcal H_C$. Moreover, we prove that such geodesic is the limit of a sequence of $C^{2,\a}_\b$ approximate geodesics under the $C^{1,1}_\b$-norm. As a geometric application, we prove the metric space structure of $\mathcal H_C$.

\end{abstract}
\tableofcontents
\section{Introduction}
We shall always denote by $X$ a smooth compact K\"ahler manifold without boundary
of complex dimension $n\geq 1$,
by $[\om_0]$ a K\"ahler class of $X$,
and by $\mathcal H$ the space of K\"ahler metrics in $[\om_0]$.
In their pioneering works, Mabuchi \cite{MR909015},
Donaldson \cite{MR1736211} and Semmes \cite{MR1165352}, independently defined the famous
Weil-Peterson type metric in $\mathcal H$ , under which
$\mathcal H$ becomes a non-positive curved infinite-dimensional symmetric space.
Semmes \cite{MR1165352} pointed out that the geodesic equation in $\mathcal H$
is a homogeneous complex Monge-Amp\`ere (HCMA) equation,
\begin{equation}\label{geo smooth}
  \left\{
   \begin{array}{ll}
{(\Om_{0}+\frac{\sqrt{-1}}{2}\p\bar\p\Psi)^{n+1}=0}&\text{ in }X\times R \; ,\\
\sum_{1\leq i,j\leq n}{(\Om_{0}+\frac{\sqrt{-1}}{2}\p\bar\p\Psi)_{i\bar j}\, dz^{i} \wedge dz^{\bar j}>0}
&\text{ in }X\times \{{z^{n+1}}\} \; ;
   \end{array}
  \right.
\end{equation}
{here $R$ is a cylinder with boundary, and}  $\Om_0$ is the pull-back metric of $\om_0$
under the natural projection. 

Geodesics are basic geometric objects in the infinity dimensional manifold $\mathcal{H}$. The intensive relation between the geodesics of $\mathcal{H}$ and the existence and the uniqueness of the cscK metrics was pointed out by Donaldson in \cite{MR1736211}.
He also conjectured {that $\mathcal{H}$ endowed with the Weil-Peterson type metric}
is geodesically convex and is a metric space. 
Chen \cite{MR1863016} established the {existence of $C^{1,1}$ geodesic {segments} (of bounded mixed derivatives)
under} smooth Dirichlet conditions and thus verified
that the space of K\"ahler metrics is a metric space. Later, Blocki \cite{BlockiZbigniew2012}
proved the $C^{1,1}$ geodesic segment has bounded Hessian when $(X\times R, \Om_0)$ has nonnegative bisectional curvature.
Phong-Sturm \cite{MR2242635}, Song-Zeltdich \cite{MR2672796}\cite{MR2394541}\cite{MR2880224}
approximated the $C^{1,1}$ geodesic
by the Bergman geodesics in finite-dimensional symmetric spaces. Later Chen and Tian in \cite{MR2434691}
improved the partial regularity of the $C^{1,1}$ geodesic,
then proved the uniqueness of the {extremal} metrics.
Donaldson \cite{MR1959581}, Darvas-Lempert \cite{Darvas:2012fk} and Lempert-Vivas \cite{Lempert:2011uq}
showed that {a} $C^{1,1}$ geodesic does not need to be smooth in general.
On the other hand, the geodesic ray induced by the test configuration is constructed in Arezzo-Tian \cite{MR2040638},
Chen-Tang \cite{MR2521647}, Phong-Sturm \cite{MR2377252}\cite{MR2661562} and Phong-Sturm \cite{MR2377252}\cite{MR2661562}.
The $C^{1,1}$ geodesic ray parallel to a given one is constructed in Chen \cite{MR2471594}
under the geometric condition {``tamed by a bounded ambient geometry''}.
We would like to remark that the existence of $C^{1,1}$ geodesic has been proved by Chen-He \cite{MR2775873}
in the space of volume {forms} on a Riemannian manifold, by P.-F. Guan-X. Zhang \cite{MR2900546} in Sasakian manifolds
and by B. Guan-Q. Li \cite{MR2673728} in Hermitian manifolds.

In this paper, our aim is to construct the natural geodesic in the moduli space of all K\"ahler metrics singular along the divisor $D$ for future study. Let us isolate now the concept, central to our aim, of K\"ahler cone metric.

\begin{defn}\label{defn: cone metrics inside introduction}
Let $X$ and  $[\om_0]$ as {at} the beginning of the paper,
and let $D=\sum^m_{i=1} {(1-\b_i)} V_i $ be a normal crossing,
effective smooth divisor of $X$ with $0 <\b_i \leq 1$ for $1\leq i\leq m$,
where $V_i \subset X$ are irreducible hypersurfaces.
Set $\b :=(\b_1,\dots,\b_m)$ and call the $\b_i$'s the \emph{cone angles}.
Given a point $p$ in $D$, label  a local chart $(U_p, z^i)$ centered at $p$
as \emph{ local cone chart} when $z^1,\dots z^k$
are the local defining functions of the hypersurfaces where $p$ locates.
A \emph{K\"ahler cone metric} $\om$ of cone angle $2\pi\b_i$ along $V_i$, for $1\leq i \leq m$,
is a closed positive $(1,1)$ current and a smooth K\"ahler metric on the regular part $M:=X\setminus D$.
{In a local cone chart $U_p$ its} K\"ahler form is quasi-isometric to the cone flat metric, which is
\begin{align}\label{flat cone}
\om_{cone}
&:=\frac{\sqrt{-1}}{2}\sum_{i=1}^k\b_i^2|z^i|^{2(\b_i-1)}dz^i\wedge
dz^{\bar i}
+\sum_{k+1\leq j\leq n}dz^j\wedge dz^{\bar j}\, .
\end{align}
\end{defn}
Let $\mathcal H_{\b}$ be the space of K\"ahler cone metrics of cone angle $2\pi\b_i$ along $V_i$ in $[\om_0]$.
It is clear that when {for all $i$ there holds $\b_i=1$, then} $\mathcal H_{\b} $ consists of all cohomologous
smooth K\"ahler metrics on a compact K\"ahler manifold.
Let {$s$} be a global meromorphic section of $[D]$.
Let $h$ be an Hermitian metric on $[D]$.
{It
is shown in Donaldson} \cite{MR2975584} that, for sufficiently small $\delta>0$,
\begin{align}\label{model cone}
\om&=\om_0+\delta {\sum_{i=1}^m \frac{\sqrt{-1}}{2}\p\bar\p|s_i|^{2\b_i}_{h_\L}}
\end{align}
is a K\"ahler cone metric.
Moreover, $\om$ is independent of the choices of $\om_0$, $h_\L$, $\delta$ up to quasi-isometry. We call it model metric in this paper.

A special K\"ahler cone metric is the K\"ahler-Einstein cone metric which is studied in many recent papers.
They have been studied in McOwen \cite{MR938672},
Troyanov \cite{MR1034288}\cite{MR1005085} for Riemannian surfaces.
The study of K\"ahler-Einstein cone metrics was initiated in Tian \cite{MR1603624}
and Tsuji \cite{MR976585} {concerning various inequalities involving the Chern numbers.}
Recently, Donaldson \cite{MR2975584} defined a new function space and
developed a program to look for the smooth K\"ahler-Einstein metric by deforming the cone angle.
Existence theorems are proved by Brendle \cite{Brendle:2011kx} for Ricci flat K\"ahler cone
metrics
, by Jeffres, Mazzeo and Rubinstein \cite{Jeffres:2011vn}
for the Fano case under the properness of the twisted K-energy
, by Campana, Guenancia and P\u{a}un \cite{Campana:2011kx} for the normal
crossing divisors
and by Berman, Boucksom, Eyssidieux, Guedj, Zeriahi \cite{Berman:2011fk} on log Fano varieties.
With the log-$\a$ invariants, Berman \cite{Berman:2010ve} solved the {existence problem for small cone angles.}
{After finishing} our paper, {more extensive developments of Donaldson's program
on the application of the K\"ahler-Einstein cone metrics to the the K\"ahler-Einstein problem have appeared};
we mention some of the most recent beautiful papers
\cite{Chen:2012mz}\cite{Chen:2012kx}\cite{Chen:2012uq}\cite{Chen:2013fk}\cite{Tian:2012fr}.

In this paper, we study the geometry of the space of K\"ahler cone metrics, in particular, the geodesic in $\mathcal H_\b$. Now we clarify the concept of geodesic in $\mathcal H_\b$.
A \emph{cone geodesic} is a curve segment $\vphi\in \mathcal H_\b$ for $0\leq t\leq 1$ which satisfies the natural generalization of the problem \eqref{geo smooth}; i.e. we are requiring that $\om_{\vphi(t)}$ is a K\"ahler cone metric for any $0\leq t\leq1$.
In this article, we find the geometric boundary conditions which assure the existence and the uniqueness of the cone geodesic.
Those lead to an appropriate {choice} of a subspace of $\mathcal{H}_{\b}$. As we will show in Section \ref{space}, the geodesic equation leads to the Dirichlet problem of the HCMA equation with the boundary potentials of cone singularities.
The Dirichlet problem of HCMA was studied intensively by many authors under various analytic boundary conditions
(see \cite{MR0445006}\cite{MR0254877}\cite{MR1935845}\cite{MR637493}\cite{MR780073}...).
In our particular environment, 
the underlying manifold is a product manifold and the curvature conditions on the background metrics play
{an important role as in the geometric-analysis problems (see the useful tricks we explain at the beginning of Section \ref{close} and Remark \ref{modifyrhs}).

The slight difference between our equation and the standard HCMA is that in our case the boundary
values allow cone singularities. So the problem is how to choose the appropriate function spaces where the solutions live in. A possible function space could be the edge space. The corresponding elliptic theory is investigated by many authors (see Mazzeo \cite{MR1133743}, Melrose \cite{MR1127161}, Schulze \cite{MR1142574} and references therein). In this edge space, Jeffres, Mazzeo and Rubinstein \cite{Jeffres:2011vn}
improved the higher regularity of the {K\"ahler-Einstein}
cone metrics. In our environment, the problem is that the edge space is defined for manifolds without boundary; which is not our case.
So, we do not use edge space in this paper.
We overcome this problem by generalizing Donaldson's space to the boundary case (see Definition \eqref{defn: 3dribdy}), that is more natural for our geometric problem. However, it would be interesting to understand whether the edge space (or with some modification) could be defined near the boundary and how to improve the regularity in such space. Finally, it is interesting to see that the cone geodesic are translated as solution of the HCMA, then the cone singularities on the boundary travel naturally to the interior of the domain.
We hope this phenomenon will be helpful to understand the solution of the complex Monge-Amp\`ere equation.

Now we specify the geometric conditions on the boundary metrics.
(The space $C_\b^{3}$ is introduced in Definition \ref{defn: 3dri}.)}
\begin{defn}\label{spaceHC}
Assume $D$ is disjoint smooth hyper surface and the cone angles $\b$
belong to the interval $(0,\frac{1}{2})$.
Then, we denote as $\mathcal H_\b^{3}$ the space of $C^{3}_\b$ $\om_0$-plurisubharmonic
potentials.
Moreover, we label as $\mathcal H_C \subset \mathcal H^{3}_\b$ one of the following spaces;
\begin{align*}
\mathfrak{I_1}&=\{\vphi\in \mathcal H^{3}_\b\text{ {such that} }  \sup Ric(\om_\vphi)
\text{ is bounded}\} ;\\
\mathfrak{I_2}&=\{\vphi\in\mathcal H^{3}_\b \text{ {such that} } \inf Ric(\om_\vphi)
\text{ is bounded}\} .
\end{align*}
\end{defn}

In general the K\"ahler cone metrics do not have bounded geometry. The Riemannian curvature of $\om$ is bounded from below when when the cone angle is less than $\frac{1}{2}$. We will compute that the Levi-Civita connection of the model cone metric defined in \eqref{model cone} under the cone coordinate (see \eqref{conetransform}) is bounded when the cone angle is less than $\frac{2}{3}$. So we need the curvature conditions of the boundary metrics to improve the regularities.
The space $\mathcal H_C$ at least contains all K\"ahler-Einstein cone metrics with the cone angle between $0$ and $\frac{1}{2}$ (see {Proposition $6.7$} in Brendle \cite{Brendle:2011kx}).
The further discussion on the properties of the subspace $\mathcal{H}_C$ will be in the forthcoming paper.
In the present work, our main aim is to prove the following result (cf. Theorem \ref{sole geo}).

\begin{thm}\label{geo existence}
Any two K\"ahler cone metrics in $\mathcal H_C$ are connected by a unique $C_\b^{1,1}$ cone geodesic. More precisely, it is the limit under the $C_\b^{1,1}$-norm by a sequence of $C^{2,\a}_{\b}$ approximate geodesics.
\end{thm}
The notion of approximate geodesic is given in \lemref{approximate}.
As {an application, we prove the following result.}
\begin{thm}
$\mathcal H_C$ is a metric space.
\end{thm}

Concerning geodesics with weak regularity,
we should compare the construction in Berndtsson's remarkable paper \cite{MR2480611}
with our result.
It is easy to compute that the volume of the K\"ahler cone metric belongs to $L^p$ with $p(\b_i-1)+1>0$
for any $1\leq i\leq k$.
According to Kolodziej's theorem in \cite{MR1618325},
there exists a unique H\"older continuous $\om_0$-{plurisubharmonic} potential.
Berndtsson \cite{MR2480611} proved that given two bounded $\om_0$-{plurisubharmonic} potentials,
there is a bounded geodesic connecting them.
Then since the advantage of using the Ding functional (cf. Ding \cite{MR967024}) is that it requires less regularity of the potentials,
as observed by Berndtsson, the convexity of the Ding functional along the bounded geodesic is applied
to prove the uniqueness of K\"ahler-Einstein cone metrics (generalizing the Bando-Mabuchi uniqueness theorem \cite{MR946233}).
However the cone geodesic {we construct} here has more {regularity} across the divisor
in a subspace $\mathcal{H}_C$ which still contains the critical metrics.
The {regularity} of the cone geodesic across the divisor are not only important to prove
the metric structure as we show in this paper,
but also to our further application {on existence and uniqueness of cscK} cone metrics.

Now we state an application {of our main theorem} to the smooth K\"ahler metrics with slightly less geometric conditions than the $C^{1,1}$ geodesic in Chen's theorem \cite{MR1863016}.
\begin{cor}\label{smooth geo existence}
If the $C^{3}$ norm and Ricci curvature upper  (or lower) bound of two K\"ahler potentials
are uniformly bounded, then the geodesic connecting them has uniform  $C^{1,1}$ bound.
\end{cor}

Now we describe the structure of our paper.
In Section {\ref{space}}, we recall the notations and the function spaces
introduced by Donaldson. In particular, we {define} the boundary case.
{Then}, we generalize the Riemannian structure to the space of K\"ahler cone metrics.
The delicate part here is the growth rate near the divisor.
In {the Donaldson} space,
we derive that the geodesic equation is a HCMA with cone singularities by integration by part
and explain the construction of the initial metric for the continuity method.

In Section \ref{close}, we obtain the \emph{a priori} estimates of the approximate Monge-Amp\`ere equation.
It is divided into several steps.
The $L^{\infty}$ estimate is derived from cone version of the maximum principle and
the super-solution of the linear equation obtained in Section \ref{open}.
In order to find out the proper geometric global conditions, the interior {Laplacian} estimate
is obtained using the techniques of Yau's second order estimate \cite{MR480350} and the Chern-Lu formula (see \cite{MR0234397}\cite{MR0250243}\cite{MR0486659}).
In order to prove the boundary Hessian estimate estimate near the divisor,
we can not use the the distance function as the barrier function which is introduced in Guan-Spruck \cite{MR1247995},
since we need a uniform estimate independent of the distance to the divisor.
So we choose the auxiliary function by solving the linear equation provided by Section \ref{open}. We hope this method could have potentially further application to Monge-Amp\`ere equation on manifold with boundary arises in other geometric problems.
In order to obtain the interior gradient estimate near the divisor,
we carefully choose an appropriate test function near the divisor.
The appropriate growth rate is important to us.

In Section \ref{open}, we solve the linearized equation
and prove the $C^{2,\a}_\b$ regularity of the approximate geodesic equation. Both the interior and the boundary Schauder estimates are of the general form. Note that the right hand side of the approximation equation \eqref{eq_starting_@_interior_holder_estimate} contains $\log\Om^{n+1}$. When applying the Evans-Krylov estimate, we need to bound the first derivative of $\log\Om^{n+1}$. We will show that it is bounded when the cone angle is less than $\frac{2}{3}$.
Thus with these estimates, the existence and the uniqueness of the $C^{1,1}_\b$ cone geodesic are proved. Moreover, the approximate geodesic is in $C^{2,\a}_\b$.

We also include an application of the interior Schauder estimate to the regularity of the K\"ahler-Einstein cone metrics (see Proposition \ref{ke schauder}). There is also a term $f$ on the right hand side of the corresponding equation \eqref{keg}. When apply the Evans-Krylov estimate, it is necessary to bound the first derivative of $f$. We show that the gradient of this term is bounded when the cone angle is less than $\frac{2}{3}$. When the cone angle is less than $\frac{1}{2}$, Brendle \cite{Brendle:2011kx} derived Calabi's three order estimate to prove the existence of Ricci flat cone metrics. 

Section \ref{linear thy} contains the maximum principle and the H\"older continuity
of the linearized equation.
In particular, the weak Hanack inequality is used to prove the $C^{2,\a}_\b$ regularity
of the approximation geodesic equation.

In Section \ref{metricstr},
we apply our cone geodesic to prove the metric structure of $\mathcal H_C$. Once we establish the $C^{1,1}_\b$ regularity of cone geodesic, the proof of the metric structure is immediate.

\noindent {\bf Acknowledgments:}
Both authors would like to thank Xiuxiong Chen who brought this problem to their attention.
The second author also thanks Claudio Arezzo for helpful discussions and ICTP for their hospitality.
He is also grateful to G\'erard Besson for his warm encouragement during his visit in Institut Fourier.
\section{The space of K\"ahler cone metrics}\label{space}
 In this section, we first introduce some notations and knowledge
of Donaldson's program \cite{MR2975584}, which we will
stick to in the remainder of the paper.
Let $U_p$ a local cone chart as in Definition \ref{defn: cone metrics inside introduction}.
{  Let $W : U_p\setminus D  \rightarrow U_p\setminus D$ }
be the change of coordinates given by
\begin{equation}\label{conetransform}
W (z^1, \cdots, \, z^n)  :=
(w^1=|z^1|^{\b_1-1}z^1,\, \cdots, \, w^k=|z^k|^{\b_k-1}z^k, \, z^{k+1}, \, \cdots, \, z^n)\; .
\end{equation}
Now, for any $1\leq i \leq k$ let $0\leq \theta_i < 2\pi$,
$z^i:=\rho_i e^{\sqrt{-1}\theta_i}$ and $r_i :=|z^i|^{\b_i}=|w^i|$;
meanwhile, for any $k+1\leq j \leq n$ let $z^j:=x^j + \sqrt{-1}y^j$.
Then, let the polar coordinate transformation of $(w^1,\cdots,\, w^k,\, z^{k+1},\cdots,\, z^n)$ be
$$
P: (w^1,\cdots,w^k,z^{k+1},\cdots,z^n) \rightarrow(r_1, \, \theta_1, \, \cdots,r_k , \,
\theta_k,x^{k+1}, \, y^{k+1}, \, \cdots, \, x^n, \, y^n)\; .
$$
Thus we obtain that the expression of the push-forward { cone} flat metric is
\begin{align}\label{euc}
((P\circ W)^{-1})_\ast g=\sum_{1\leq i\leq k}[dr_i^2+\b_i^2r_i^2d\theta_i^2]
+\sum_{k+1\leq j\leq n}[(dx^j)^2+(dy^j)^2]\; .
\end{align}
This flat metric is uniformly equivalent to the standard Euclidean metric.
However, letting $\mu_i :=\b_i^{-1}-1$, we have $w^i=r_ie^{\sqrt{-1}\theta_i}=|w^{i}|^{-\mu_i}z^i$;
moreover, we define
\begin{align*}
\varepsilon_i&:=dr_i+\sqrt{-1}\b_i r_id\theta_i
=\b_i|w^i|^{1-\mu_i}{(w^i)^{-1}} dz^i\\
&=\b_i
\left[
\left(
1+\frac{\mu_i}{2}
\right)
|w^i|{(w^i)^{-1}} dw^i+\frac{\mu_i}{2}|w^i|^{-1}w^idw^{\bar i}
\right] \, ,
 \end{align*}
and we notice that it is not a holomorphic $1$-form, since $\p_{ w^{\bar i}}\varepsilon_i\neq0$.
Consequently, $\varepsilon_i$ and $dz^j$ merely form a local orthonormal basis of the $(1,0)$-forms.

Now we present the function spaces which are introduced by Donaldson in
\cite{MR2975584}.
The H\"older space $C^{\a}_{\b}$  consists in those functions $f$
which are H\"older continuous with respect to a K\"ahler cone metric.
Also, $C^{\a}_{\b,0}$ denotes  the subspace of those functions in $C^{\a}_\b$
for which their limit is zero along $V_i$ for any $1\leq i \leq m$.
The H\"older continuous $(1,0)$-forms, in local coordinates $U_p$,
can be expressed as
$$\xi = f_i\varepsilon_i+f_jdz^j \, , $$
where the Einstein notation is adopted, $f_i \in C^{\a}_{0}$ and $f_j\in C^{\a}$.
Meanwhile, a H\"older $(1,1)$-form $\eta$ in local coordinates $U_p$ is of the shape
$$
\eta = f_{i_1\bar{i_2}}\varepsilon_{i_1}\varepsilon_{\bar{i_2}}+f_{i\bar j}
\varepsilon_idz^{\bar j}
+f_{\bar i j}\varepsilon_{\bar i}dz^j+f_{j_1\bar {j_{2}}}dz^{j_1}dz^{\bar {j_{2}}}\;;
$$
here the coefficients satisfy
{  $f_{i\bar j}, f_{\bar i j} \in C^{\a}_{0}\; $
and $f_{i_1\bar{i_2}},f_{j_1\bar{j_{2}}} \in C^{\a} \; .$}
Note that according to this Definition,
for any K\"ahler cone metric $\om\in C^\a_\b$, around the point $p\in D$,
we have a local normal coordinate such that $g_{ij}(p)=\delta_{ij}$.
\begin{defn}\label{defn: twice diff mixed derivatives functions}
The H\"older space $C^{2,\a}_{\b}$ is defined by
\begin{align*}
C^{2,\a}_{\b}
= \{f\; \vert \;  f, \p f, \p\bar\p f\in C_\b^{\a}\}\; .
\end{align*}
\end{defn}
Note that the $C^{2,\a}_\b$ space, since it  concerns only the mixed derivatives,
is different from the usual $C^{2,\a}$ space.
The definition of higher order space $C_\b^{k,\a}$ depends on the background metrics. In this paper, we use the flat cone metric $\om_{cone}$ \eqref{flat cone} to define $C_\b^{k,\a}$.
It is not hard to see that, by the quasi-isometric mapping $W$, $\p\bar\p f\in C_\b^{\a}$ is equivalent to $\frac{\p^2}{\p w^i\p w^{\bar j}}\in C^\a$ for any $1\leq i,j\leq n$ under the coordinate $\{w^i\}$. So we say the third derivative of a function belongs to $C^{\a}_{\b}$ if $$\frac{\p^3}{\p w^k \p w^i\p w^{\bar j}} f\in C^\a$$ for any $1\leq i,j,k\leq n$. In particular,
\begin{defn}\label{defn: 3dri}
The H\"older space $C^{3}$ is defined by
\begin{align*}
C^{3}_{\b}
= \{f \vert  f\in C_\b^{2,\a} \text{ and the third derivative of $f$ w.r.t $\om_{cone}$ is bounded}\}\; .
\end{align*}
\end{defn}
Thus the higher order spaces are defined by induction on the index $k$.
Now we postpone the discussion of the function space for a while, we will continue after we introduce the product manifold where the geodesic equation is defined. 

We then approach some considerations on the Riemannian geometry of the space of K\"ahler cone metrics. Recall that $\mathcal H_\b^{2,\a}$ is the space of $C^{2,\a}_\b$ $\om_0$ psh-functions. It is a convex, open set in $C^{2,\a}_\b$.
The tangent space of $\mathcal H_\b^{2,\a}$ at a point $\vphi$ is $C^{2,\a}_\b$. We generalize the
Donaldson \cite{MR1736211}, Mabuchi \cite{MR909015}, Semmes \cite{MR1165352}
metric to $\mathcal H_\b^{2,\a}$ by associating to $\vphi\in\mathcal H_\b^{2,\a}$
and tangent vectors $\psi_1 , \psi_2 \in T_\vphi \mathcal H_\b^{2,\a}$, the
real number
\begin{equation}\label{definitionofmetric}
\int_M \psi_1 \cdot \psi_2 \om^n_\vphi \; .
\end{equation}
The definition makes sense for  {  K\"ahler cone  metrics},
since the volume of the K\"ahler cone metrics is finite.
Furthermore, we choose  { an arbitrary differentiable} path $\vphi\in C^1([0,1],\mathcal H_{\b}^{2,\a})$
and {  along it, differentiable vector field} $\psi\in C^1([0,1],C_\b^{2,\a})$.
We thus define { the following  derivation of the vector field on $M = X\backslash D$}
\begin{equation}\label{definitionofconnection}
D_t \psi := \frac{\partial \psi}{\partial t} -
(\p\psi , \p \frac{\p\vphi} {\p t})_{ g_\vphi} \; {.}
\end{equation}
 We claim that \eqref{definitionofconnection} is the Levi-Civita connection of \eqref{definitionofmetric}.
The fact that \eqref{definitionofconnection} is torsion free
comes from a point-wise computation on $M$.
  Thus, the claim will be accomplished after  verifying the metric compatibility,
which is done in Proposition \ref{prop:metric compatibility}.
We first prove an integration by parts formula.
\begin{lem}\label{lemma: integration by parts}
Assume that $\vphi_1$, $|\p \vphi_1|_\Om$, $|\p \vphi_2|_\Om$, $|\Delta \vphi_2|_{L^1(\Om)}$
are bounded. Then the following formula holds
\begin{align*}
 \int_M \vphi_1 \Delta \vphi_2 \om^n = -\int_M (\p \vphi_1 , \p \vphi_2 )_\Om \om^n \; .
\end{align*}
\end{lem}
\begin{proof}
 Choose a cut-off function $\chi_\eps$ which vanishes in a neighborhood of $D$.
Then,
\begin{align*}
 \int_M \chi_\eps \vphi_1 \Delta \vphi_2 \om^n
= - \int_M \chi_\eps (\p \vphi_1 , \p \vphi_2)_\Om \om^n
- \int_M \vphi_1 (\p \chi_\eps , \p \vphi_2)_\Om \om^n \; .
\end{align*}
The convergence of the first two terms { follows} from the  Lebesgue dominated convergence theorem.
So, it suffices to find a $\chi_\eps$ such that $\int_M |\p \chi_\eps|_\Om \om^n \to 0$
as $\eps \to 0$.\\
Choose $\chi_\eps := \chi (\frac{1}{\eps^2} \prod_{i} |s_i|^2)$, where
$s_i$ are the defining functions of $D_i$ and
$\chi$ is a smooth non-decreasing function  such that
\begin{align*}
 \left\{
\begin{array}{ll}
\chi = 0 & \mbox{ in } [0,1]\\
0\leq \chi \leq 1 & \mbox{ in } [1,2]\\
\chi = 1 & \mbox{ in } [2 , +\infty) \; .\\
\end{array}
\right.
\end{align*}
Now,
\begin{align*}
 |\p \chi_\eps|_\Om \leq \chi ' \cdot \frac{1}{\eps^2} |s_i| |s_i|^{1-\b_i}
= \frac{C}{\eps^2} |s_i|^{2-\b_i} \; .
\end{align*}
So, as $\eps \to 0 $ we get in the cone chart
\begin{align*}
 \int_M |\p \chi_\eps |_\Om \om^n
&\leq \int_{|s|=r}\int_0^{2\pi}\int_\eps^{2\eps} \frac{2\pi}{\eps^2} |r|^{2-\b_i} |r|^{2(\b_i-1)} r dr d\theta dz^2\wedge\cdots\wedge dz^n\\
&\leq \frac{2\pi}{\eps^2} \int_\eps^{2\eps} |r|^{1+\b_i} dr
\leq C \eps^{\b_i} \to 0 \; .
\end{align*}
This completes the proof of the lemma.
\end{proof}
As an application of the above formula, we have
\begin{prop}\label{prop:metric compatibility}
The connection \eqref{definitionofconnection} is compatible with
the metric \eqref{definitionofmetric}.
\end{prop}
\begin{proof}
 We compute
\begin{align*}
 \frac{1}{2} \frac{d}{dt}\int_M \psi^2 \om_\vphi^n
=\frac{1}{2}\int_M  (2\psi \psi ' + \psi^2 \Delta_\vphi \vphi ') \om_\vphi^n \; .
\end{align*}
Since $\psi^2$, $|\p(\psi^2)|_{g_\vphi}$, $|\p \vphi '|_{g_\vphi}$, $\Delta_\vphi \vphi '$
are all bounded with respect to $g_\vphi$, we are allowed to apply
 Lemma \ref{lemma: integration by parts}
and we get
\begin{align*}
 \frac{1}{2} \frac{d}{dt}\int_M \psi^2 \om_\vphi^n
= \frac{1}{2} \int_M [2\psi \psi ' -2 \psi (\p \psi , \p \vphi ' )_{ g_\vphi}] \om_\vphi^n \; .
\end{align*}
This completes the proof of the proposition.
\end{proof}

Next, we derive the geodesic equation.
\begin{prop}
The geodesic equation satisfies the following equation on $M$ point-wise
\begin{align}\label{geo}
\vphi''-(\p\vphi',\p\vphi')_{g_{\vphi}}=0 \; .
\end{align}
\end{prop}
\begin{proof}
Assume that $\vphi(t)|_0^1$ is a path from $\vphi_0$ to $\vphi_1$,
and that $\vphi( s, t)\in C^1([0,1]\times[0,1],C_\b^{2,\a})$ is
a $1$-parameter variation of $\vphi(t)|_0^1$ with fixed endpoints.
We minimize the length function
$$
L(\vphi(s , t))=
\int_0^1\sqrt{\int_M \left(\frac{\partial \vphi(s, t)}{\partial t}\right)^2\om^n_\vphi}\quad dt \;.
$$
We are going to compute the variation of $\frac{\partial }{\partial s}L(\vphi(s, t))$;
denote $\vphi ' =  \frac{\partial \vphi}{\partial t}$ and
$$
E = \int_M \vphi'^2\om^n_\vphi \; .
$$
Then, using \eqref{definitionofconnection} and the compatibility property
we get
\begin{align*}
\frac{\partial }{\partial s} L (\phi(s ,t)) &=
\int_0^1\frac{1}{E} \int_M D_s \vphi' \cdot \vphi ' \om^n_\vphi dt
=   \int_0^1 \frac{1}{E}  \int_M D_t \frac{\partial}{\partial s} \vphi \cdot \vphi '
 \om^n_\vphi dt
\\
&=  \int_0^1 \frac{1}{E} \left[  \frac{\partial}{\partial t} \int_M \frac{\partial}{\partial s}
\vphi\cdot \vphi' \om^n_\vphi  - \int_M D_t \frac{\partial}{\partial s} \vphi \cdot \vphi '
\om^n_\vphi \right] dt\\
&= -\int_0^1\frac{1}{E}  \int_M \frac{\partial}{\partial s} \vphi  \cdot D_t \vphi'
 \om^n_\vphi dt  \; .
\end{align*}
The first term in the second line vanishes since the endpoints are fixed.
So the geodesic condition reads
\begin{align*}
0 = \frac{\partial }{\partial s} L (\phi(s ,t)) =  -\int_0^1\frac{1}{E}
 \int_M \frac{\partial}{\partial s} \vphi  \cdot D_t \vphi' \om^n_\vphi dt
\end{align*}
which implies that the geodesic equation is
\begin{align*}
D_t \vphi ' \equiv 0 \mbox{ on } M \; .
\end{align*}
%
\end{proof}


Consider the cylinder $R=[0,1]\times S^1$ and introduce
the coordinate $z^{n+1}=x^{n+1}+\sqrt{-1}y^{n+1}$ on $R$.
Define  $$\vphi(z',z^{n+1})=\vphi(z^1,\cdots,z^n,x^{n+1})=\vphi(z^1,\cdots,z^n,t)$$
on the product manifold $X\times R$ and  let  $\pi$ be the natural projection form $X\times R$ to $X$.
We also denote
\begin{align*}
z&=(z',z^{n+1})=(z^1,\cdots,z^n,z^{n+1}) \; ,\\
\Om_0&=(\pi^{-1})^\ast \om_0+dz^{n+1} \wedge d\bar z^{n+1} \; ,\\
\Om&=(\pi^{-1})^\ast \om+dz^{n+1} \wedge d\bar z^{n+1} \; ,\\
\Psi&=\vphi-|z^{n+1}|^2 \; .
\end{align*}
It is a matter of algebra to show that \eqref{geo} could be reduced to a
degenerate Monge-Amp\`ere equation.  A path $\vphi(t)$
 with endpoints $\vphi_0 $, $\vphi_1$  satisfies the geodesic equation \eqref{geo} on $X$
if
and only if $\Psi$ satisfies  the following  Dirichlet problem
involving a degenerate complex Monge-Amp\`{e}re  equation
\begin{equation}\label{geo ma}
  \left\{
   \begin{array}{ll}
\det(\Om_{i\bar j}+\Psi_{i\bar j})=0&\text{ in }M\times R \; ,\\
\Psi(z)=\Psi_0&\text{ on } X \times \p R  \; ,\\
\sum_{1\leq i,j\leq n}(\Om_{i\bar j}+\Psi_{i\bar j})dz^{i}dz^{\bar j}>0
&\text{ in }X\times \{z^{n-1}\} \; .
   \end{array}
  \right.
\end{equation}
Here the  following  Dirichlet boundary  conditions $\Psi_0$ are satisfied
\begin{equation}\label{geo ma boundary}
  \left\{
   \begin{array}{ll}
\Psi_0(z',0)\doteq\Psi(z',\sqrt{-1}y^{n+1})={ \vphi_0(z')}-(y^{n+1})^2 &\text{ on }
X\times \{0\}\times S^1,\\
\Psi_0(z',1)\doteq\Psi(z',1+\sqrt{-1}y^{n+1})={ \vphi_1(z')}-1-(y^{n+1})^2 &\text{ on }
X\times \{1\}\times S^1.
   \end{array}
  \right.
\end{equation}

Now we are given a $n+1$-dimensional K\"ahler manifold 
$\mathfrak{X} = X \times R$ with boundary;  
the given data of the Dirichlet problem are put on two disjoint copies of $X$. 
We also have a divisor $\mathfrak{D}=D\times R$, with $D$ as in Definition \eqref{defn: cone metrics inside introduction}, which intersects transversely the boundary
 Let  $f_i$ be the local defining function  of each irreducible analytic component
$V_i$ of $\mathfrak{D}$.
Then the transition functions $\frac{f_i}{f_j}$ give a line bundle $[\mathfrak{D}]$ in $\mathfrak{X}$.
Let {$s$ be a} global meromorphic section of $[\mathfrak{D}]$.
Let $h_\L$ be the Hermitian metric on $[\mathfrak{D}]$.
{There is a small positive $\delta$ such that, on} $\mathfrak{X}$,
\begin{align}\label{Om}
\Om=\Om_0+\delta { \sum_{i=1}^m \frac{\sqrt{-1}}{2}\p\bar\p|s_i|^{2\b_i}_{h_\L}}
\end{align}
is a K\"ahler cone metric {(cf. \eqref{model cone}).}
Moreover, it is also independent of the choices $\Om_0$, $h_\L$, $\delta$ up to quasi-isometry.

We could define the H\"{o}lder space $C^{3}_\b$ in the interior of $(\mathfrak{X}, \mathfrak{D})$
as that one defined on $(X,D)$. On the boundary, near a point $p$ we choose a local holomorphic coordinate $\{U^+_p; z^i=x^i+iy^i\}$, $1\leq i\leq 2n+2$ centered at $p$. From the discussion above, we see that the boundary of $\mathfrak X$ is $x^{n+1}=0$. When $U_p^+$ does not intersect the divisor $\mathfrak D$, the H\"older space is defined in the usual way. So it is sufficient to defined a new H\"older space in the coordinates which contain the points of the divisor. We first note that the solution of geodesic equation is independent of the variable $y^{n+1}$, so the partial derivative on the variable $x^{n+1}$ is the same to the one on the variable $z^{n+1}$. Next, the quasi-isometric mapping $W$ is still well defined in $U_p^+$ as follows, 
\begin{equation*}
W (z^1, \cdots, \, z^{n+1})  :=
(w^1=|z^1|^{\b_1-1}z^1,\, \cdots, \, w^k=|z^k|^{\b_k-1}z^k, \, z^{k+1}, \, \cdots, \, z^{n+1})\; .
\end{equation*}
So we could define the H\"older space $C^{\a}_{\b}(U_p^+)$ to be the set of   functions 
which are H\"older continuous under $\{z^i\}_{i=1}^{n+1}$ with respect to a K\"ahler cone metric.
Also, $C^{\a}_{\b,0}(U_p^+)$ denotes  the subspace of those functions in $C^{\a}_\b(U_p^+)$
for which their limit is zero along $V_i$ for any $1\leq i \leq m$.
The H\"older continuous $(1,0)$-forms, in local coordinates $U_p^+$,
can be expressed as
$$\xi =\sum_{i} f_i\varepsilon_i+ \sum_j f_jdz^j \, , $$
where $f_i \in C^{\a}_{0}(U_p^+)$ and $f_j\in C^{\a}(U_p^+)$.
Meanwhile, a H\"older $(1,1)$-form $\eta$ in local coordinates $U_p^+$ is of the shape
$$
\eta = f_{i_1\bar{i_2}}\varepsilon_{i_1}\varepsilon_{\bar{i_2}}+f_{i\bar j}
\varepsilon_idz^{\bar j}
+f_{\bar i j}\varepsilon_{\bar i}dz^j+f_{j_1\bar {j_{2}}}dz^{j_1}dz^{\bar {j_{2}}}\;;
$$
here the coefficients satisfy
{  $f_{i\bar j}, f_{\bar i j} \in C^{\a}_{0}(U_p^+)\; $
and $f_{i_1\bar{i_2}},f_{j_1\bar{j_{2}}} \in C^{\a}(U_p^+) \; .$}
The H\"older space $C^{2,\a}_{\b}(U_p^+)$ is parallelly defined by
\begin{align*}
C^{2,\a}_{\b}(U_p^+)
= \{f\; \vert \;  f, \p f, \p\bar\p f\in C_\b^{\a}(U_p^+)\}\; .
\end{align*}
Then we use the flat cone metric $\om_{cone}$ \eqref{flat cone} to define the higher order space $C_\b^{k,\a}(U_p^+)$. 
The boundary $C^3$ space is defined in the same manner.
\begin{defn}\label{defn: 3dribdy}
The H\"older space $C^{3}(U_p^+)$ is defined by
\begin{align*}
C^{3}_{\b}(U_p^+)
= \{f \vert  f\in C_\b^{2,\a}(U_p^+) \text{ and the 3nd derivative of $f$ w.r.t $\om_{cone}$ is bounded}\}\; .
\end{align*}
\end{defn}
Thus the higher order spaces are also defined by induction on the index $k$ in the same way.

In order to apply the maximum principle, we require that the maximum point does not live on the divisor. The following lemma by {Jeffres} \cite{MR1800816} is used to overcome this trouble. With the discussion above, we prove this technical auxiliary lemma in our product manifold $\mathfrak{X}$ with boundary.

\begin{lem}\label{auxiliary function}
There is a positive constant $\kappa$ such that $S=||s||^{2\kappa}$ satisfies the
following properties
\begin{enumerate}
  \item $\frac{\sqrt{-1}}{2}\p\bar\p S\geq\kappa\frac{\sqrt{-1}}{2}\p\bar\p\log||s||^2\geq -C\Om$,
  \item  for  any $\a>0$, $S$ grows faster than the collapsing of $\Phi\in C^{\a}_\b$ near $\mathfrak{D}$ provided $2\kappa\leq a\b$.
  \end{enumerate}
\end{lem}
\begin{proof}
Since
\begin{align*}
\frac{\sqrt{-1}}{2}\p\bar\p S
&=\frac{\sqrt{-1}}{2}S(\kappa\p\bar\p\log ||s||^2+\p\log S\wedge\bar\p\log S)\\
&\geq\kappa\frac{\sqrt{-1}}{2}S\p\bar\p\log ||s||^2,
\end{align*}
and since $-\frac{\sqrt{-1}}{2}\p\bar\p\log ||s||^2$ is the curvature form of the line bundle
under the Hermitian metric $h$,
there is a constant $C$ such that $\frac{\sqrt{-1}}{2}\p\bar\p\log S\geq -C\Om$.
So we have
\begin{align*}
\frac{\sqrt{-1}}{2}S\p\bar\p\log S\geq -C\kappa S\Om\geq -C\kappa\Om.
\end{align*}
In order to derive the second conclusion, we compute the first derivative of $S$ along
the singular direction.
Choosing the basis $e$, we have $||s||^{2\kappa}=|z|^{2\kappa}||e||^{2\kappa}$,
then the main term of $|\nabla^a S|^2_\Om$ is $|z^1|^{4\kappa-2a+2a(1-\b)}$.
So it is sufficient to choose $\kappa$ such that this main term becomes unbounded
as it approaches $\mathfrak{D}$.
Meanwhile, $\Psi\in C^{a}_\b$ implies that $|\nabla^a\Psi|_{\Om}$ is bounded,
so the second conclusion follows.
\end{proof}

In order to apply the continuity method we first construct the starting metric of the solution
path such that it satisfies the boundary conditions.
Since $\Psi_0$ may not be convex along the direction $\frac{\p }{\p z^{n+1}}$,
we have to extend $\Psi_0$ to whole $\mathfrak{X}$ as follows.
Let $\tilde{\Psi}_0$ be the line segment between the boundary K\"ahler cone potentials
  $\vphi_0$ and   $\vphi_1$;
namely,  (cf. \eqref{geo ma boundary})
\begin{align*}
\tilde{\Psi}_0=t{  \Psi_0 (z',1)} +(1-t){  \Psi_0(z',0) } +t+(y^{n+1})^2
=
t{  \vphi_1} +(1-t){  \vphi_0} \; .
\end{align*}
Then we choose a function   $\Phi$  which depends only on $z^{n+1}$ such that
\begin{equation*}
  \left\{
   \begin{array}{ll}
\Phi(z^{n+1})=0 &\text{ on } \p \mathfrak{X} \, ,\\
\Phi_{z^{n+1}\bar z^{n+1}}>0 &\text{ in } \mathfrak{X} \; .
   \end{array}
  \right.
\end{equation*}
We denote the new potential by $$\Psi_1 : =\tilde{\Psi}_0+m\Phi.$$
Next we verify that   $\Psi_1$ is a K\"ahler cone potential on $\mathfrak{X}$.
\begin{prop}\label{csck eqv}
 Suppose that $\vphi_0,\vphi_1\in \mathcal{H_\b}$. Then there  exists a large number $m$ such that
\begin{align}\label{Om1}
\Om_{1}:=\Om+\frac{\sqrt{-1}}{2}\sum_{1\leq i,j\leq n+1}\p_i\p_{\bar j}\Psi_1
\end{align} is a K\"ahler cone metric on $(\mathfrak{X}, \mathfrak{D})$.
\end{prop}
\begin{proof}
The local expression of $\Om_{\Psi_1}$ is
\begin{align*}
&\Om+\frac{\sqrt{-1}}{2}\sum_{1\leq i,j\leq n+1}\p_i\p_{\bar j}(\tilde{\Psi}_0+m\Phi)\\
&=t\om_{{  \vphi_1 } }
+(1-t)\om_{{  \vphi_0}}+\frac{\sqrt{-1}}{2}(1+m\p_{n+1}\p_{\overline{n+1}}\Phi)dz^{n+1} \wedge d\bar z^{n+1}\\
&+\frac{1}{\sqrt{2}}(\p_i {  \vphi_1}-\p_i {  \vphi_0} )dz^idz^{\overline{n+1}}
+\frac{1}{\sqrt{2}}(\p_{\bar i}{  \vphi_1}-\p_{\bar i}{  \vphi_0})dz^{\bar i}dz^{n+1} \; .
\end{align*}
We call  $\om_t : =t\om_{  \vphi_1}+(1-t)\om_{  \vphi_0}$  the line segment
and $\psi :=  \vphi_1-  \vphi_0$ the difference  of the boundary K\"ahler cone potentials.

In order to show that $\Om_{\Psi_1}$
is a K\"ahler cone metric on $\mathfrak{X}$,
it suffices to verify two conditions; that it is positive on the regular part
$M$ and that $\Om_{\Psi_1}$ is locally quasi-isometric to
\begin{align*}
\Om_{cone}&=\frac{\sqrt{-1}}{2}\sum_{i=1}^k(\b_i^2|z^i|^{2(\b_i-1)}dz^i\wedge
dz^{\bar i})+\sum_{i=k+1}^{n+1}(dz^i\wedge dz^{\bar i}) \; .
\end{align*}
Since the determinant  of $\Om_{\Psi_1}$ is
$det(g_t)[1+m\Phi_{n+1,\overline{n+1}}-g_t^{i\bar j} \psi_i\psi_{\bar j}]$,
the former condition is true once we choose $m$ large enough.
The latter condition is verified   as $\vphi_0,\vphi_1\in \mathcal{H_\b}$.
\end{proof}



\section{A priori estimates}\label{close}
In this section, we derive uniform  \emph{a priori} estimates for the degenerate equation.
With the same background as (2.6), $\mathfrak{M} = M \times R$ and $\Psi_1$ a K\"ahler potential in $\mathfrak{M}$, that is $\Om_{1}:= \Om + \frac{\sqrt{-1}}{2}\p\bar\p \Psi_1>0$,
we consider the family of Dirichlet problems for $0\leq\tau\leq1$,
\begin{equation}\label{per equ}
  \left\{
   \begin{array}{ll}
\det(\Om_{i\bar j}+\Psi_{i\bar j})
=\tau  e^{\Psi-\Psi_1}\det(\Om_{i\bar j}+\Psi_{1i\bar j}) &\text{ in }\mathfrak{M}\;,\\
\Psi(z)=\Psi_0 &\text{ on }\p\mathfrak{X}\;,
   \end{array}
  \right.
\end{equation}
in the space $C^{2,\a}_\b$.
We will specify the conditions on $\Psi_0$ in each estimate.

Since the curvature conditions of the background metrics are required when we derive the a priori estimates, we explain  an  observation on how to choose appropriate background metrics.
If we take $\Om_{1}$ as the background metric, we obtain an equivalent equation
\begin{equation}\label{per equ sim}
  \left\{
   \begin{array}{ll}
\det({\Om_{1 i\bar j}}+\tilde\Psi_{i\bar j})
=\tau f e^{\tilde\Psi} \det(\Om_{i\bar j})
=\tau e^{\tilde\Psi} \det({\Om_{1i\bar j}}) &\text{ in }\mathfrak{M}\; ,\\
\tilde\Psi(z)=0 & \text{ on }\p\mathfrak{X}\; ,
   \end{array}
  \right.
\end{equation}
where $\tilde\Psi:= \Psi- \Psi_1$ and
$f:=\frac{\det(\Om_{i\bar j}+\Psi_{1 i\bar j})}{\det(\Om_{i\bar j})}$.
In general,  given a K\"ahler cone potential $\Phi$
we could take $\Om^A:=\Om+\frac{\sqrt{-1}}{2}\p\bar\p\Phi$,
$\Psi^A:=\Psi-\Phi$, $\Psi_1^A:=\Psi_1-\Phi$, $\Psi_0^A:=\Psi_0-\Phi$.
The new family of Dirichlet problems becomes
\begin{equation}\label{new equ}
  \left\{
   \begin{array}{ll}
\det(\Om^A_{i\bar j}+\Psi^A_{i\bar j})
=\tau  e^{\Psi^A-\Psi^A_1}\det(\Om^A_{i\bar j}+\Psi^A_{1i\bar j}) &\text{ in }\mathfrak{M}\;,\\
\Psi(z)=\Psi_0^A &\text{ on }\p\mathfrak{X}\;.
   \end{array}
  \right.
\end{equation}
The above observation will be particularly useful when we  will  derive the \emph{a priori} estimates later.
Note that the right hand side of the equation is positive as long as $\tau$ is positive.
When $\tau=1$, $\Psi_1$ solves the equation.
When $\tau$ is zero, \eqref{per equ} as well as \eqref{per equ sim}  provide
a solution of the degenerate equation \eqref{geo ma}.
\subsection{$L^{\infty}$ estimate}\label{l infty} We will see in the following that the $L^{\infty}$
estimate follows from the cone maximum principle (\lemref{maximum_principle}) and the global bounded weak solution (Proposition \ref{glob bound claim}) provided in Section \ref{linear thy}.
Applying the logarithm on both sides of \eqref{per equ sim} we have
\begin{align}\label{log equ}
\log\frac{\det(\Om_{1i\bar j}+\tilde\Psi_{i\bar j})}{\det(\Om_{1i\bar j})}=\log\tau +\tilde\Psi \;.
\end{align}
\begin{prop}\label{low}(Lower bound of $\Psi$)
For any point $x\in \mathfrak{X}$, the following estimate holds
$$
\Psi(x)\geq\Psi_1(x)\; .
$$
\end{prop}
\begin{proof}
According to the second conclusion of \lemref{auxiliary function},
$U=\tilde\Psi-\eps S$ achieves its minimum point $p$ on $\mathfrak{M}$.
There are two cases, one  when  $p$ is on $M\times \p R$ and the other one when $p$
is in the interior of $\mathfrak{M}$. In the first case,
since $p$ is on the regular part of the boundary, then the minimal
value is just the boundary value. Thus the inequality holds
automatically. Now we explain the second case.
The equation \eqref{log equ} is rewritten as
\begin{align}\label{log new}
\log\frac{(\Om_1+\frac{\sqrt{-1}}{2}\p\bar\p(U+\eps S))^{n+1}}
{\Om_1^{n+1}}=\log\tau+\tilde\Psi \; .
\end{align}
At the point $p$ the Hessian of $U$ is non-negative $U_{i\bar i}\geq 0$;
so, after diagonalizing $\Om_1$ and $\Om_1+\frac{\sqrt{-1}}{2}\p\bar\p(U+\eps S)$
simultaneously, \eqref{log new} implies
\begin{align*}
\tau e^{\tilde\Psi(p)}\Pi_{i=1}^{n+1}\Om_{1i\bar i}
\geq\Pi_{i=1}^{n+1}(\Om_{1i\bar i}+\eps S_{i\bar i})
\geq (1-\eps C)^{n+1}\Pi_{i=1}^{n+1}\Om_{1i\bar i} \;,
\end{align*}
where, at the second inequality, we use the first conclusion of \lemref{auxiliary function}.
Then we have
$$
\tilde\Psi(p)\geq\log (1-\eps C)^{n+1} \; .
$$
Then for any $x\in \mathfrak{X}$, $(1)$ in \lemref{auxiliary function} implies
\begin{align*}
&\tilde\Psi(x)=U(x)+\eps S(x)
\geq U(p)\\
&= \tilde\Psi(p)-\eps S(p)
\geq \log (1-\eps C)^{n+1}-\eps \; ,
\end{align*}
which gives the lower bound of $\tilde\Psi$ as $\eps$ goes to zero.
\end{proof}

\begin{prop}\label{upp}(Upper bound of $\Psi$)
For any point $x\in \mathfrak{X}$, the following estimate holds
\begin{align*}
\Psi(x)\leq h(x) \; .
\end{align*}
\end{prop}
\begin{proof}
From \eqref{per equ} the solution is non-negative $\Om+\frac{\sqrt{-1}}{2}\p\bar\p\Psi\geq0$,
after taking trace it implies
$$
-\tri \Psi\leq n+1 \; .
$$
In order to obtain the lower bound, we then consider the linear equation
\begin{equation*}
  \left\{
   \begin{array}{ll}
\tri h =-n-1 & \text{ in }\mathfrak{M}\; ,\\
h=\Psi_0 & \text{ on }\p\mathfrak{X} \;.
   \end{array}
  \right.
\end{equation*}
It is solvable by Proposition \ref{ext weak} and \ref{glob bound claim}.
Then the lemma follows form the weak maximum principle of cone metrics (\lemref{wmp}).
\end{proof}
\begin{rem}\label{modifyrhs}
We could consider the family of equations with parameter $a\in \mathbb{R}$ as
\begin{equation}\label{per equ a}
  \left\{
   \begin{array}{ll}
\det(\Om_{1i\bar j}+\tilde\Psi_{i\bar j})
=\tau e^{a\tilde\Psi} \det(\Om_{1i\bar j}) &\text{ in }\mathfrak{M}\;,\\
\tilde\Psi(z)=0 &\text{ on }\p\mathfrak{X}\;.
   \end{array}
  \right.
\end{equation}
The approximate equation \eqref{per equ sim} is the former with $a=1$.
That is slightly different from the family considered by Chen \cite{MR1863016}
with $a=0$.
We would like to indicate that using the estimate in Section \ref{linear thy},
the lower bound of the solution of Chen's approximate equation
can be proved by applying the maximum principle with respect to the K\"ahler cone metric \lemref{wmp} to
\begin{equation*}
  \left\{
   \begin{array}{ll}
\det(\Om_{1i\bar j}+\tilde\Psi_{i\bar j})\leq \det(\Om_1)&\text{ in }\mathfrak{M}\;,\\
\tilde\Psi=0 &\text{ on }\p \mathfrak{X}\;.
   \end{array}
  \right.
\end{equation*}
\end{rem}
The upper and the lower bound of $\Psi$ imply the boundary gradient estimate
\begin{align}\label{prop_boundary_gradient_estimate}
\sup_{M\times \p R}|\nabla\Psi|_{\Om}\leq \sup_{\mathfrak{X}}
|\nabla\Psi_1|_{\Om}+\sup_{\mathfrak{X}}|\nabla h|_{\Om}\; .
\end{align}

\subsection{Interior Laplacian estimate}
The content of the present subsection is
the statement and proof of three different interior Laplacian estimates
(Proposition \ref{sec}).

We remark that in \lemref{lem_1_interior_laplacian_estimate} below,
we could choose different background metrics.
As a result, constants would have different dependence { on geometric quantities}.
\begin{prop}\label{sec}
There are three constants $C_i$, for $i=1,2,3$ such that
\begin{align}\label{claim of subsection interior laplacian estimate}
\sup_{\mathfrak{X}}(n+1+\tri\Psi)\leq C_i\sup_{\mathfrak{\p X}}(n+1+\tri\Psi)\; .
\end{align}
The  constants respectively depend on
\begin{align*}
C_1&=C_1(\, \inf Riem(\Om),\,\sup \Ric(\Om_1), \, \sup \tr_{\Om}\Om_1,\,
\osc\Psi,\,\osc\Psi_1)\; ;\\
C_2&=C_2(\, |Riem(\Om_1)|_{L^\infty},\, \sup \tr_{\Om}\Om_1,\, \sup\tr_{\Om_1}\Om,\, \osc\Psi)\; ;\\
C_3&=C_3(\, \sup Riem(\Om),\, \inf \Ric(\Om_1),\, \sup \tr_{\Om}\Om_1, \, \osc\Psi, \, \osc\Psi_1)\;.
\end{align*}
\end{prop}
\begin{rem}
The estimates work for any given K\"ahler cone metric $\Om$.
\end{rem}
We first consider the equation \eqref{per equ}. We denote
\begin{align}\label{eq 0 interior laplacian estimate}
F{:=}\log\tau+\log f+\Psi-\Psi_1=\log\frac{\det(\Om_{i\bar j}+\Psi_{i\bar j})}{\det(\Om_{i\bar j})}\;.
\end{align}
We calculate $\Delta ' (n+1 + \Delta \Psi) $ of our equation and explain later how to change the background metric.
\begin{lem}\label{lem_1_interior_laplacian_estimate}
 The following formula holds
\begin{align*}
 \Delta ' (n+1 + \Delta \Psi) &= g^{i\bar j}{g'}^{k \bar l}{g'}^{p \bar q}
\partial_{\bar l} {g'}_{p \bar j} \partial_{k}{g'}_{i \bar q} -\tr_{\Omega} \Ric (\Omega_1)\\
&+\Delta \Psi - \Delta \Psi_1
+ {g'}^{k \bar l}R^{i\bar j}_{\phantom{i\bar j}k\bar l}{g'}_{i \bar j}\; .
\end{align*}
\end{lem}
\begin{proof}
 Since ${g'}_{i\bar j} = g_{i\bar j } + \Psi_{i\bar j}$,
when we take $-\partial_k \partial_{\bar l}$ on both sides we get
\begin{align}\label{eq 1 interior laplacian estimate}
 -\partial_k \partial_{\bar l} {g'}_{i\bar j} =
R_{i\bar j k \bar l}
- \Psi_{i\bar j k \bar l}\; .
\end{align}
Since the Riemannian curvature is defined by
\begin{align*}
 {R'}_{i\bar j k \bar l} =
-\partial_k \partial_{\bar l} {g'}_{i \bar j}
+ {g'}^{p\bar q}\partial_{\bar l}{g'}_{p\bar j} \partial_k {g'}_{i \bar q},
\end{align*}
inserting the latter in \eqref{eq 1 interior laplacian estimate} and taking the trace
with respect to ${g'}^{k\bar l}$ and ${g}^{i \bar j}$ we have
\begin{align}\label{eq 2 interior laplacian estimate}
 g^{i\bar j} {R'}_{i\bar j} =
g^{i \bar j}{g'}^{k \bar l}{g'}^{p\bar q}
\partial_{\bar l}{g'}_{p\bar j} \partial_k {g'}_{i \bar q}
+ {g'}^{k\bar l}R_{k\bar l} - g^{i\bar j} {g'}^{k \bar l}\Psi_{i\bar j k \bar l}\; .
\end{align}
 Since
\begin{align*}
 \Delta ' (n+ 1+\Delta \Psi)
= {g'}^{k \bar l}g^{i \bar l}\Psi_{i\bar j k \bar l}
+ {g'}^{k \bar l}R^{i \bar j}_{\phantom{i \bar j}k\bar l}\Psi_{i\bar j}\; ,
\end{align*}
inserting the latter in \eqref{eq 2 interior laplacian estimate} we get
\begin{align*}
 \Delta ' (n+ 1+\Delta \Psi)
=
g^{i \bar j}{g'}^{k \bar l}{g'}^{p\bar q}
\partial_{\bar l}{g'}_{p\bar j} \partial_k {g'}_{i \bar q}
-g^{i\bar j} {R'}_{i\bar j}
+{g'}^{k\bar l}R_{k\bar l}
+ {g'}^{k \bar l}R^{i \bar j}_{\phantom{i \bar j}k\bar l}\Psi_{i\bar j}\; .
\end{align*}
Since \eqref{eq 0 interior laplacian estimate} implies
${R'}_{i\bar j} = R_{i\bar j} - F_{i \bar j}$ we therefore have
\begin{align*}
 \Delta' (n+1 + \Delta \Psi) =
g^{i \bar j}{g'}^{k \bar l}{g'}^{p\bar q}
\partial_{\bar l}{g'}_{p\bar j} \partial_k {g'}_{i \bar q}
-S(\Omega) + \Delta F + {g'}^{k \bar l}R^{i \bar j}_{\phantom{i \bar j}k\bar l}{g'}_{i\bar j}\; .
\end{align*}
Then the lemma follows from the formula
\begin{align*}
 \Delta F = \Delta (\log f + \Psi - \Psi_1) = - \tr_{\Omega}\Ric(\Omega_1) +S(\Omega)+
\Delta \Psi - \Delta \Psi_1\; .
\end{align*}
This completes the proof of the lemma.
\end{proof}
\smallskip
The following formula follows from the Schwarz inequality.
See Yau \cite{MR480350}, and Siu \cite{MR942521} page 73.
\begin{align}\label{lem_2_interior_laplacian_estimate}
 g^{i \bar j}{g'}^{k \bar l}{g'}^{p\bar q}
\partial_{\bar l}{g'}_{p\bar j} \partial_k {g'}_{i \bar q}
\geq
\frac{|\partial (n+1 + \Delta \Psi)|^2}{n+1 + \Delta \Psi}\; .
\end{align}
\begin{lem}
There is a constant $C$ depending on $\sup \Ric(\Omega_1), \, \sup \tr_{\Omega}\Omega_1 ,
\, \inf_{i\neq k}R_{i\bar i k \bar k} $ such that
\begin{align*}
 \Delta' (\log (n+1+\Delta\Psi))\geq -C(1+ \tr_{\Omega '}\Omega).
\end{align*}
\end{lem}
\begin{proof}
We compute
\begin{align*}
 \Delta ' (\log (n+1 + \Delta \Psi))=
\frac{\Delta ' (n+1 + \Delta \Psi)}{n+1+\Delta\Psi}
-\frac{|\partial (n+1 + \Delta \Psi)|^2}{n+1 + \Delta \Psi}\; .
\end{align*}
Thus, by combining Lemma \ref{lem_1_interior_laplacian_estimate} with
{\eqref{lem_2_interior_laplacian_estimate}}, we have
\begin{align*}
\Delta ' (\log (n+1 + \Delta \Psi))
&\geq
\frac{-\tr_{\Omega}\Ric(\Omega_1) + \Delta \Psi - \Delta \Psi_1
+ {g'}^{k \bar l}R^{i\bar j}_{\phantom{i\bar j}k\bar l}{g'}_{i \bar j}}{n+1+\Delta\Psi}\\
&\geq -C (1+ \frac{1}{n+1+\Delta \Psi} + \tr_{\Omega '} \Omega)\; .
\end{align*}
Thus the lemma follows from $\frac{1}{n+1+\Delta\Psi} \leq \frac{1}{1+\Psi_{i\bar i}}\leq \tr_{\Omega '}\Omega$.
\end{proof}
\begin{proof}(proof of constant $C_1$)
Denote
\begin{align*}
 Z:= \log(n+1 + \Delta \Psi) - K\Psi + \eps S \; ,
\end{align*}with $K$ to be chosen.
According to \lemref{auxiliary function}, with appropriate $\kappa$, the maximum point $p$ of $Z$
stays in the interior of $\mathfrak{M}$.
Since
$\Delta ' \Psi = n+1 -\tr_{\Omega '}\Omega$, and $\Delta ' S \geq -C\tr_{\Omega '}\Omega$ (\lemref{auxiliary function}),
then at $p$ there holds
\begin{align*}
 0\geq  \Delta ' Z
\geq -C (1+\tr_{\Omega ' }\Omega ) - K (n+1-\tr_{\Omega '}\Omega) -\eps C \tr_{\Omega ' }\Omega\;.
\end{align*}
Now we choose $K$ such that $-C + K - \eps C > 0$ to obtain the upper bound of $\tr_{\Omega '}\Omega (p)$.
From the arithmetic-geometric-mean inequality we have
\begin{align*}
 (n+1 + \Delta\Psi)^{\frac{1}{n}} \cdot e^{-\frac{F}{n}}&=
\left(
\sum_{i=1}^{n+1}\frac{1}{\prod_{k=1,k\neq i}^{n+1}(1+\Psi_{k\bar k})}
\right)^{\frac{1}{n}}\\
&\leq \sum_{k=1}^{n+1}\frac{1}{1+\Psi_{k\bar k}}=\tr_{\Omega '}\Omega\;.
\end{align*}
Since $F=\log \tau + \log f + \Psi - \Psi_1$,
so, $n+1+\Delta \Psi$ is bounded from above at $p$ depending on
$\sup \Ric(\Omega_1)$, $ \sup \tr_{\Omega}\Omega_1$, $ \inf_{i\neq k}R_{i\bar i k \bar k}$,
$\sup \Psi $, and $ \inf \Psi_1  $.
For any $x\in \mathfrak{X}$, there holds
$Z(x)\leq \sup_{\partial\mathfrak{X}} Z + Z(p)$.
Hence,
\begin{align}\label{got first constant for interior lapl est}
  n+1+\Delta\Psi &= e^{Z+K\Psi - \eps S}\nonumber\\
&\leq e^{\sup_{\partial \mathfrak{X}} Z + Z(p) + K \sup \Psi} \nonumber\\
&\leq \sup_{\partial \mathfrak{X}} (n+1+\Delta\Psi)
e^{-K\inf_{\mathfrak{X}}\Psi_0 + 1 + Z(p) + K\sup \Psi}\; .
\end{align}
This formula gives precisely the claimed inequality
\eqref{claim of subsection interior laplacian estimate} for the first constant $C_1$.
\end{proof}
\begin{proof}(proof of constant $C_2$)
Now the same argument as in Lemma \ref{lem_1_interior_laplacian_estimate},
applied to equation \eqref{per equ sim},
gives the following formula
\begin{align*}
 \Delta ' (n+1+\Delta_1 \tilde{\Psi})
= {g_1}^{i \bar j } {g'}^{k\bar l}{g'}^{p\bar q}
\partial_{\bar l}{g'}_{p\bar j} \partial_{k} {g'}_{i\bar q}
-S(\Omega_1) + \Delta_1 \tilde{\Psi} + {g'}^{k\bar l}{R_1}^{i\bar j}_{\phantom{i\bar j}k\bar l}{g'}_{i \bar j}\;.
\end{align*}

Then, still following an argument similar to that used in the first part of this subsection,
we get a constant $C$ which depends on $\sup S(\Omega_1)$,
$\inf_{i\neq k} R_{i\bar i k \bar k}(\Omega_1)$, $\osc \Psi$, such that
\begin{align*}
 n+1+\Delta_1 \tilde{\Psi} \leq C \sup_{\partial\mathfrak{X}} (n+1+\Delta_1 \tilde{\Psi})\;.
\end{align*}
 Since $\Omega$ and $\Omega_1$ are $L^{\infty}$ equivalent, we have
\begin{align}\label{got second constant for interior lapl est}
 n+1+\Delta\Psi \leq C(\sup \tr_{\Omega_1}\Omega)(\sup \tr_{\Omega}\Omega_1)\cdot
\sup_{\partial \mathfrak{X}}(n+1+\Delta\Psi)\; .
\end{align}
This formula gives precisely the second constant $C_2$ for claimed inequality
\eqref{claim of subsection interior laplacian estimate}.
Here the conditions $\inf Riem(\Om_1)$ and $\sup S(\Om_1)$
are bounded are equivalent to the $L^\infty$ bound of the Riemannian curvature of $\Om_1$.
\end{proof}
\begin{proof}(proof of constant $C_3$)
Now we use the Chern-Lu formula (see \cite{MR0234397}\cite{MR0250243}\cite{MR0486659}) to derive the second order estimate.
We get the formula of $$\tr_{\Omega '}\Omega = n+1- \Delta ' \Psi.$$
This following identity is interpreted as the energy identity of the harmonic map $id$ between
$(M,g')$ to $(M,g)$.
\begin{align}\label{lem_3_interior_laplacian_estimate}
 \Delta' (\tr_{\Omega '}\Omega)=
{R'}^{i\bar j}g_{i\bar j} - {g'}^{i\bar j}{g'}^{k\bar l}R_{i\bar j k \bar l}
-g^{i \bar j}{g'}^{k \bar l}{g'}^{p\bar q}
\partial_{\bar l}{g'}_{p\bar j} \partial_k {g'}_{i \bar q}\; .
\end{align}
The Schwarz inequality implies
\begin{align}\label{la sch}
g'^{k\bar l}\p_kg'^{i\bar j}g_{i\bar j}\p_{\bar l}g'^{p\bar q}
g_{p\bar q}
\leq
-(g'^{k\bar l}g'^{p\bar j}g_{i\bar j}\p_{\bar l}
g'_{p\bar q}\p_{k}g'^{i\bar q})(g'^{i\bar j}g_{i\bar j})\;.
\end{align}
Now we use the equation \eqref{log equ}.
\begin{lem}\label{la sch two}
The following formula holds
\begin{align*}
\tri'(\log tr_{\Om'}\Om)\geq
-(n+1) - C(\tr_{\Omega '}\Omega),
\end{align*}
with $C$ that depends on $\inf \Ric(\Omega_1)$,
$\sup_{i\neq k} R_{i\bar i k\bar k} (\Omega)$,
$\sup \tr_{\Omega}\Omega_1$.
\end{lem}
\begin{proof}
We apply \eqref{lem_3_interior_laplacian_estimate} and \eqref{la sch} to obtain
\begin{align*}
\tri'[\log tr_{\Om'}\Om]
&=\frac{\tri'(tr_{\Om'}\Om)}{tr_{\Om'}\Om}-
\frac{g'^{k\bar l}\p_kg'^{i\bar j}g_{i\bar j}\p_{\bar l}g'^{p\bar q}
g_{p\bar q}}{(tr_{\Om'}\Om)^2}\\
&\geq\frac{R'^{i\bar j}g_{i\bar j}-g'^{i\bar j}g'^{k\bar l}R_{i\bar jk\bar l}}{tr_{\Om'}\Om}\;.
\end{align*}
From \eqref{log equ} we have
\begin{align*}
\Ric'&=\Ric(\Om)-\frac{\sqrt{-1}}{2}\p\bar\p F\\
&=\Ric(\Om) + \Ric(\Om_1) -\Ric(\Om) -\Om'+\Om_1\;,
\end{align*}
then $\Ric' \geq (\inf \Ric(\Om_1) +1)\Om_1 -\Om '$ and
\begin{align*}
 {R'}^{i\bar j}g_{i\bar j}
&\geq
(\inf \Ric(\Omega_1) +1){g'}^{i\bar l}{g'}^{k\bar j}g_{1,k\bar l}g_{1, i\bar j}
- {g'}^{i\bar l}{g'}^{k\bar l}{g'}_{k\bar l}g_{i\bar j}\\
&\geq
-C (\tr_{\Omega '}\Omega)^2 \cdot (\tr_{\Omega}\Omega_1)^2
-(n+1) \tr_{\Omega}\Omega_1\; ,
\end{align*}
where $C$ is a positive constant depending on $\inf \Ric (\Omega_1)$.
Then we have
\begin{align*}
 \Delta ' (\log \tr_{\Omega ' }\Omega ) \geq -(n+1)-C(\tr_{\Om '}\Om)\; ,
\end{align*}
where $C$ depends on $\inf \Ric(\Om_1)$, $\sup_{i\neq k}R_{i\bar i k \bar k}$,
$\sup \tr_{\Om} \Om_1$.
This completes the proof of the lemma.
\end{proof}

Consider $Z_1 := \log \tr_{\Om '}\Om - C' \Psi + \eps S$, such it
has a maximum point $p$ which stays away from $\mathfrak{D}$, and with $C'$ to be chosen.
Then
\begin{align*}
 \Delta ' Z_1 \geq -(n+1)-C\tr_{\Om '}\Om - C' ((n+1) - \tr_{\Om '}\Om) -C\eps \tr_{\Om'}\Om\;.
\end{align*}
Now we choose $C'$ such that $C'-C-C\eps >0$ and we have at $p$, $\tr_{\Om '}\Om \leq C$.
In the same vein as the first part of the subsection we compute that for any $x\in
\mathfrak{X}$ there holds
\begin{align*}
 \log \tr_{\Omega '}\Om (x) &= Z_1 (x) + C' \Psi -\eps S +
\sup_{\partial \mathfrak{X}}\log \tr_{\Om '}\Om \\
&\leq Z_1 (p) + C' \sup \Psi + \sup_{\partial\mathfrak{X}}\log \tr_{\Om '}\Om
\end{align*}
Using the arithmetic-geometric-mean inequality we have
\begin{align}\label{got third}
 (\tr_{\Om  }\Om ')^{\frac{1}{n}} \leq \tr_{\Om '}\Om e^{\frac{F}{n}}
\leq C \sup_{\partial \mathfrak{X}} \tr_{\Om ' }\Om \;,
\end{align}
where $C$ depends on $\inf \Ric (\Om_1)$, $\sup_{i\neq k} R_{i\bar i k \bar k} (\Om)$,
$\sup \tr_{\Om}\Om_1$, $\osc \Psi$, $\inf \Psi_1$.
This formula gives precisely the third constant of formula
\eqref{claim of subsection interior laplacian estimate}.
\end{proof}
We could choose $\Om_1$ as the background metric and repeat the estimate, but it will not provide more information.
The three constants $C_i$
are determined by the formulas \eqref{got first constant for interior lapl est},
\eqref{got second constant for interior lapl est} and
\eqref{got third}, respectively.
This concludes the proof of Proposition \ref{sec}.
\subsection{Boundary Hessian estimate}
The boundary hessian estimate for real and complex Monge-Amp\`ere equation is developed
in \cite{MR780073}, \cite{MR1168119}, \cite{MR1247995},\cite{MR1664889} and \cite{MR1863016}.
The difficulty {that arises} in our problem is the estimate near the singular varieties $V_i$.
The distance function can not be used in our problem,
since we need the uniform estimate which is independent of the distance to the divisor $\mathfrak D$.
We overcome this difficulty by multiplying singular terms with proper weight and using the linear theory
developed in Section \ref{linear thy} to construct an appropriate barrier function which is independent
of the distance function.
\begin{prop}\label{prop: boundary hessian estimate}
The following boundary estimate holds
\begin{align*}
{\sup_{ X\times \p R}	|\sqrt{-1}\p\bar\p\Psi|_\Om \leq C(\sup_{\mathfrak X} |\p\Psi|_\Om+1)\; .}
\end{align*}
The constant $C$ depends on
$|\p\tilde g_{1\a\bar\b}|$,\,
$|\Psi|$,\, $|\p \Psi_1|_\Om$,\, $|\p \Psi_0|_\Om$,\, $|\p\bar\p \Psi_0|$,\,
\end{prop}
\begin{proof}
Fix a point {$p\in M\times \p R $,
and consider $U_p\subset M \times R$} an open neighborhood of $p$.
Recall that we denote by $\Psi$ an \emph{a priori} solution of the equation
$$
\det(\Om_{i\bar j}+\Psi_{i\bar j})
=\tau f e^{\Psi-\Psi_1} \det(\Om_{i\bar j}) \; ,
$$
whose
boundary values are given by the datum $\Psi_0$.
The tangent-tangent term of the boundary Hessian estimate follows from the boundary value
directly. Since the boundary is flat, the normal-normal term follows from the construction
of the approximate geodesic equation
$$
[\vphi''-(\p\vphi',\p\vphi')_{g_{\vphi}}]
\det\om_\vphi=\Om_\Psi^{n+1}=\tau e^{\Psi-\Psi_1} \det(\Om_{1i\bar j})\, ,
$$
i.e.
\begin{align}\label{eqn_claim_bdry_hess_estimate}
 \left|
 \frac{\partial }{\partial z^{n+1}}
\frac{\partial }{\partial z^{{\overline{n+1}}} }
\Psi
 \right|_{\Om; X\times \p R}
 \leq
\sum_i  \left|
 \frac{\partial }{\partial z^{n+1} }
\frac{\partial }{\partial z^{\overline{i}} }
{\Psi}
\right|_{\Om; X\times \p R}
+
C
\; .
\end{align}
The constant $C$ depends on $|\Psi_1|$,\, $|\Psi_0|$,\, $|\p\bar\p \Psi_0|_\Om$ and $\det(\Om_{1i\bar j})$.
The quantity $\det(\Om_{1i\bar j})$ depends on the boundary value and the chosen function $\Phi$
in {Proposition \ref{csck eqv}}.
Then the aim of the present subsection is to derive the mixed tangent-normal estimate on the boundary.

We put
$$
\tri':= \sum_{\a , \b =1}^{n+1} {g'}^{\a \overline{\b}}
\frac{\partial^2}{\partial z^{\a} \partial z^{\overline{\b}}}\; .
$$
The elliptic operator $\tri'$ allows the use of the maximum principle in Section \ref{maximum_principle}.

Our idea is to construct {a barrier function} and apply the maximum principle locally in a small neighborhood
of {the point $p\in M\times \p R$.}
Since the second order derivatives of $\Psi$ blow up near the singular points where
$\mathfrak D$ intersects {$X\times \p R$}, we need to prove that the estimates do not depend on
the choice of the diameter of the small neighborhood {$U_p$.}

Let us suppose {that the open neighborhood $U_p \subset X\times R$ is a coordinate chart near $p$
(cf. Definition \ref{defn: cone metrics inside introduction}) for  the first $n$ variables;
moreover, the coordinate $z^{n+1} := x+ \sqrt{-1}y$ in $U_p$ locally parametrizes the Riemann surface $R$.}
Next, let us define the function {$v : U_p \rightarrow \mathbb{R}$} as
\begin{align}\label{eq_defn_function_v}
 v:= (\Psi -\Psi_1) + sx - Nx^2\, ,
\end{align}
where $N,s$ are constants which depend only on
$M\times R$, the background metric $g$ and the datum $\Psi_0$, and they
will be determined later in \eqref{choice_of_N} and \eqref{choice_of_s} respectively.
Also, let us fix the small neighborhood of the origin
$\Omega_{\delta} := (M\times R)\cap B_{\delta}(0){\subset U_p}$
with small {radius} $\delta<1$.
We require that $\Omega_{\delta}$ does not intersect  $\mathfrak D$.
We will show that the estimate does not depend on the choice of $\delta$.

We first prove the following lemma.
\begin{lem}
The following inequalities hold
\begin{align}\label{eq_subclaim_inequalities_on_v}
\left\{
\begin{array}{l}
\tri'v \leq - \frac{\epsilon}{4}
\left(
1 + \sum_{\a , \b=1}^{n+1} {g'}^{\a \bar \b } g_{ \a \bar \b}\right)   \mbox{ in } \Omega_\delta
 \\
  v\geq 0 \qquad \mbox{ on } \partial \Omega_\delta \;,
\end{array}
\right.
\end{align}
where $\epsilon>0$ is a constant depending on the lower bound of $\Om_{\Psi_0}$.
\end{lem}
\begin{proof}
By means of the  equation  \eqref{eq_defn_function_v} and the linearity of $\tri'$,
let us first consider the term $\tri'(\Psi -\Psi_1)$.
Here the remark to do is that,
as the metric $g_{\a \bar \b } + \Psi_{1\a \bar \b}$ is $L^\infty$ equivalent to $g_{\a \bar \b}$ in
$X\times R$, then we can find a uniform constant $\epsilon$ such that
$g_{\a \bar \b} + \Psi_{1\a \bar \b} > \epsilon g_{\a \bar \b}$ holds point-wise in
$\Omega_{\delta} $ (could be in the whole $\mathfrak X$).
Notice that the lower bound of $\Om_{\Psi_1}$ depends on the lower bound of $\Om_{\Psi_0}$.
We conclude, using the remark, that just by definition there holds
\begin{align*}
 &\tri'(\Psi - \Psi_1) =
\sum_{\a , \bar \b=1}^{n+1} {g'}^{\a \bar \b}
\left[
\left(
g_{\a \bar \b} +\Psi_{\a \bar \b }
\right)
-
\left(
g_{\a \bar \b } + \Psi_{1 \a \bar \b}
\right)
\right]\\
&= n+1 - \sum_{\a , \bar \b=1}^{n+1} {g'}^{\a \bar \b } g_{{\Psi_1} \a \bar \b}
\leq n+1 - \epsilon \sum_{\a , \bar \b=1}^{n+1} {g'}^{\a \bar \b } g_{ \a \bar \b}\; .
\end{align*}
It is clear that $\tri'x =0$ and $\tri' x^2=2{g'}^{(n+1)\overline{n+1}}$.
Thus, we have
\begin{align*}
\tri' v &= \tri'(\Psi - \Psi_1) + s\tri'x -2N {g'}^{(n+1)\overline{n+1}}\\
&\leq n+1 - \epsilon  \sum_{\a , \bar \b=1}^{n+1} {g'}^{\a \bar \b } g_{\a \bar \b}
-2N {g'}^{(n+1)\overline{n+1}} \\
&= n+1 + \left(
-\frac{\epsilon}{2} \right)
\sum_{\a , \bar \b=1}^{n+1} {g'}^{\a \bar \b } g_{\a \bar \b}
-2N {g'}^{(n+1)\overline{n+1}}
- \frac{\epsilon}{2}  \sum_{\a , \bar \b=1}^{n+1} {g'}^{\a \bar \b }g_{\a \bar \b}\;.
\end{align*}
Without loss of generality we can prove the inequality in the local normal coordinate such that,
at the origin, {there holds} $g_{\a \bar \b} = \delta_{\a \bar \b}$. We have, at the origin,
\begin{align*}
  & N {g'}^{(n+1)\overline{n+1}} + \frac{\epsilon}{4}
\sum_{\a , \bar \b=1}^{n+1} {g'}^{\a \bar \b }g_{\a \bar \b} \\
&=
\left(
N + \frac{\epsilon}{4}
\right)
\frac{1}{1+  \Psi_{(n+1)\overline{n+1}}}
+ \frac{\epsilon}{4}
\sum_{j=1}^n \frac{1}{1+\Psi_{j\bar j}}\\
&\geq (n+1)
\left[
\left(
N+\frac{\epsilon}{4}
\right)
\cdot
\left(
\frac{\epsilon}{4}
\right)^n
\prod_{\a=1}^{n+1}
\frac{1}{1 + \Psi_{\a \bar \a}}
\right]^{\frac{1}{n+1}}\;.
\end{align*}
Since, still at the origin, there holds
\[
 \prod_{\a=1}^{n+1} \frac{1}{1+\Psi_{\a\bar\a}} = \frac{\Omega^{n+1}}{\Omega_\Psi^{n+1}}
= \frac{1}{
\tau f}e^{\Psi_1-\Psi} \geq e^{\inf(\Psi_1-\Psi)}\frac{1}{\sup f}\; ,
\]
then we choose the constant $N$ large enough so that
\begin{align}\label{choice_of_N}
 -2(n+1)
\left[
\left(
N + \frac{\epsilon}{4}
\right)
\left(
\frac{\epsilon}{4}
\right)^n
e^{\inf(\Psi_1-\Psi)}\frac{1}{\sup f}
\right]^{\frac{1}{n+1}}
+n+1 < -\frac{\epsilon}{4}\;.
\end{align}
Here $N$ depends on $\inf(\Psi_1-\Psi)$, $\sup f=\sup \frac{\Om_1^{n+1}}{\Om^{n+1}}$ and $\eps$.

To fully achieve the claim \eqref{eq_subclaim_inequalities_on_v}, we have to verify
the condition on $\partial\Omega_{\delta}$.
On $\partial \Omega_\delta \cap \partial (M\times R)$, there holds $v=0$.
On $\partial \Omega_\delta \cap \Int(M\times R)$, there holds, since $ \Psi\geq  \Psi_1$,
$v\geq (s-Nx)x \geq (s-N\delta)x$. So we choose $s=2N$ such that
\begin{align}\label{choice_of_s}
 (s-N\delta)x \geq 0.
\end{align} This completes the proof of the lemma.
\end{proof}
Now, we come to construct the auxiliary function $u$.
We construct a nonnegative boundary value $\phi$ such that $\phi$ only vanishes on {the} point $p$.
For example, $\phi=\Psi_0-\Psi_0(p)+e^{|\Psi_0-\Psi_0(p)|}-1$.
Then we solve the equation $\tri_g u_\parallel=-n-1$ with the boundary value $\phi$.
According to the maximum principle for the cone metrics {(cf. Proposition \ref{wmp}},
we have $u_\parallel\geq 0$.
Meanwhile, we choose a smooth nonnegative function $u_\perp$ of $z^{n+1}$ monotonic along $\frac{\p}{\p z^{n+1}}$
such that it vanishes on the boundary and strictly larger than $u_\parallel+1$ in the interior of $\mathfrak X$,
since $u_\parallel$ is bounded.
Now, we define the function $u$ {by adding up $u_\parallel$ and $u_\perp$.}

We need to change the variables {via the map  $W$ defined at \eqref{conetransform},
extended as the identity on the variable $z^{n+1}$;
we mark functions and operators transformed under $W$ with $\tilde{}$ on the top.
Finally, under $W$ coordinate functions become, for $1\leq i \leq n$, $w^i = x^i + \sqrt{-1}y^i$;
then,} we define $D_i := \frac{\partial }{\partial x^i}$, for $1\leq i \leq 2n$.
With the above notations, we define the function {$h : U_p \rightarrow \mathbb{R}$} as
\begin{align*}
 h:= \l_1 \tilde v + \l_2 \tilde u+ \l_3  \cdot D_i (\tilde\Psi - \tilde\Psi_1)\;,
\end{align*}
for {one fixed} $1\leq i \leq n$ and three constants $\l_1$, $\l_2$ and $\l_3$ determined below.

{We emphasize that till the end of the subsection, the index $1\leq i \leq n$ is fixed;
we recall that the cone angle $\b_i$ is equal to one for the directions corresponding to $k+1 \leq i\leq n$. }

We notice that at the origin (or point $p$), the value of $h$ is zero.
We define $\rho_i$ as the distance from $p$ to the divisor only along the coordinate {$w^i$.}
{We shrink $\Om_\delta$} to be the set containing such points whose distance to $p$
{less than half the distance from $p$ to $D$.}
{So, } on $\partial \Omega_\delta \cap \partial (M\times R)$
there holds $\frac{\rho_i}{2}\leq|{w^i}|\leq 2\rho_i$ and ${\tilde u}\geq 1$;
{then, letting} $\l_3$ be the smallest eigenvalue of the inverse matrix of $W_\ast\Om$,
there holds { for  $q\in \partial \Omega_\delta \cap \Int(M\times R)$}
\begin{align*}
h(q)&\geq \l_2 \tilde u(q)-\l_3 |D_i (\tilde\Psi - \tilde\Psi_1)(q)|\\
&\geq \l_2-C|\p(\Psi - \Psi_1)(q)|_\Om\\
&\geq 0 \, ,
\end{align*}
where the last inequality is true provided $\l_2=1+C|\p (\Psi - \Psi_1)(q)|_\Om$
with $C$ that depends on background metric $\Om$.
Let us come to analyze $\tilde\tri' h$.
\begin{lem}
There exist $\l_1$ depending on $\l_2$,
$\l_4=|D_i\log\tilde\Om_1^{n+1}|+ |\p\Psi|_{\Om}
+ |\p\Psi_1|_{\Om}$,
$\l_5=|D_i \tilde g_{1\a\bar\b}|_{\Om}$, such that
\begin{align*}
 \tilde\tri'h\leq 0.
\end{align*}
\end{lem}
\begin{proof}
By our preliminary work, we read off
\eqref{eq_subclaim_inequalities_on_v} an estimate for $ \tri'v=\tilde\tri' \tilde v$.
About $ \tilde\tri'\tilde u$, we compute
\begin{align}\label{eqn: firs for h}
  {\tilde\tri'\tilde u =} \tri' u = \sum_{\a , \b =1}^{n+1} {g'}^{\a\bar\b} u_{\a\bar\b}
  \leq C \sum_{\a , \b =1}^{n+1} {g'}^{\a\bar\b} g_{\a\bar\b}\, ,
\end{align}
{ where  $C$ is a constant depending on $\Psi_0$ and $u_\perp$.
Finally, as $\Psi$ is a solution to  $\Omega_\Psi^{n+1} = e^F\Omega^{n+1}$
with $F=\log \tau + \log f + \Psi - \Psi_1$, we differentiate this equation
under coordinate ${w^i}$, and we get }
\begin{align*}
& \tilde\tri' D_i(\tilde\Psi- \tilde\Psi_1)\\ &= \sum_{\a , \b =1}^{n+1} D_i\log\tilde\Om_1^{n+1}
+ D_i \tilde\Psi-D_i\tilde\Psi_1- \sum_{\a , \b =1}^{n+1} {(\tilde{g})'}^{\a\bar\b}D_i \tilde{g}_{1\a\bar\b}\;.
\end{align*}
{We end up with the estimate for $\tilde\tri'  D_i (\tilde\Psi -\tilde\Psi_0)$,}
\begin{align}\label{eqn: second for h}
\tilde \tri'  D_i(\tilde\Psi - \tilde\Psi_1)
\leq
\l_4
+ \l_5\sum_{\a , \b =1}^{n+1} {g'}^{\a\bar\b} g_{\a\bar\b}
\;.
\end{align}
{There,} $\l_4:= |\p\log\tilde\Om_1^{n+1}|+ |\p\Psi|_{\Om}
+ |\p\Psi_1|_{\Om}$ ,
$\l_5:=|\p \tilde g_{1\a\bar\b}|_{\Om}$.
We conclude the following estimate for $\tilde\tri' h$
{by means of \eqref{eqn: firs for h} and \eqref{eqn: second for h};}
\begin{align*}
 &\tilde\tri'  h=\l_1\tilde\tri' \tilde v+\l_2\tilde\tri' \tilde u+\l_3\tilde\tri' D_i(\tilde\Psi- \tilde\Psi_1) \\ &\leq -\l_1\frac{\epsilon}{4}
\left(
1+ \sum_{\a , \b =1}^{n+1} {g}'^{\a\bar\b} g_{\a\bar\b}
\right)
+ \l_2 {\cdot C} \sum_{\a , \b =1}^{n+1} {g'}^{\a\bar\b} g_{\a\bar\b}
+
\left[
\l_4+ \l_5 \sum_{\a , \b =1}^{n+1} {g'}^{\a\bar\b} g_{\a\bar\b}
\right] \\
&\leq
\left[
-\frac{\epsilon}{4}\l_1 +{\l_2 \cdot C} + \l_4+\l_5
\right]
\cdot
\left(
1+ \sum_{\a , \b =1}^{n+1} {g'}^{\a\bar\b} g_{\a\bar\b}
\right)<0 \;,
\end{align*}
after choosing $\l_1$ properly.
\end{proof}
{(Completion of the proof of Proposition \ref{prop: boundary hessian estimate}.)}
To summarize, we get $h\geq0$ on $\partial \Omega_\delta $ and $\tilde\tri' h <0$ in $\Omega_\delta$ in the weak sense.
So, by the weak maximum principle, we get that $h\geq 0$ in $\Omega_\delta$. Since $h(0)=0$,
then we have {(recall $z^{n+1}= x+\sqrt{-1}\, y$)}
\[
 \frac{\partial h}{\partial x}(0) \geq 0 , \quad \frac{\partial h}{\partial y}(0) \geq 0\; .
\]
In particular, we compute
\[
 \frac{\partial h}{\partial x}= \l_1 \frac{\partial(\tilde\Psi - \tilde\Psi_1)}{\partial x}+s
 -2Nx + \l_2 \frac{\partial \tilde u}{\partial x}
+ \l_3 \frac{\partial}{\partial x} D_i (\tilde\Psi - \tilde\Psi_1) \;,
\]
which leads to
\begin{align*}
\l_3 \frac{\p }{\p x} D_i (\tilde\Psi - \tilde\Psi_1)(0) \geq -s
 -\l_1 \frac{\p(\tilde\Psi - \tilde\Psi_1)}{\p x}(0)-\l_2 \frac{\p \tilde u}{\partial x}(0) \;.
\end{align*}
Combining the above inequality with $\frac{\p}{\p y}D_i (\tilde\Psi- \tilde\Psi_1) =0$,
and adding the inequalities, we get that for any $1\leq i \leq n$
\begin{align*}
\frac{\p}{\p z^{n+1}} \frac{\p}{\p {w^{\bar i}}} (\tilde\Psi-\tilde\Psi_1)(0)
\geq
-C\;,
\end{align*}
where $C$ depends on $\l_1$, $\l_2$, $|\p\Psi|_{\Om}$, $|\p\Psi_1|_{\Om}$ and $|\p u|_{\Om}$.
We repeat the same argument for $D_i = - \frac{\p}{\p x^i}$ and for
$D_i = - \frac{\partial}{\partial y^i}$ and we conclude that the tangent-normal
derivative is bounded, for $1\leq i \leq n$, by
\begin{align}\label{end_section_hessian_bdry_estimate}
 { \left| \frac{\p}{\p z^{n+1}}  \frac{\p}{\p z^{\bar i}} \Psi \right|_\Om(0)
 = \left| \frac{\p}{\p z^{n+1}}  \frac{\p}{\p w^{ \bar i}} \Psi \right|(0)
\geq
\left| \frac{\partial}{\partial z^{n+1}} \frac{\partial}{\partial w^{ \bar i}} \Psi_0 \right|(0)
+ C\;,}
\end{align}
where again
$C$ depends on $\l_1$, $\l_2$, $|\p\Psi|_{\Om}$, $|\p\Psi_1|_{\Om}$ and $|\p u|_{\Om}$.
{Note from the construction of $\Psi_1$ that the derivatives of $\Psi_1$ are}
controlled by the corresponding derivatives of $\Psi_0$.
As \eqref{end_section_hessian_bdry_estimate}
clearly coincides with \eqref{eqn_claim_bdry_hess_estimate},
this completes the proof of the proposition.
\end{proof}


\subsection{Interior gradient estimate}
We directly calculate the norm of the gradient to obtain { the differential inequality
in Proposition \ref{prop: gradient estimate}.}
{Gradient estimates were} obtained by Cherrier and Hanani \cite{MR921559}\cite{MR1383012}
for Hermitian manifolds and later by Blocki \cite{MR2495772} 
for the K\"ahler case.
Since \eqref{per equ sim} has singularity along the divisor,
in order to apply the maximum principle,
we need to choose {an} appropriate test function near the divisor.

{We define the following functions,
where $\epsilon >0$ and $\gamma: \mathbb{R}\rightarrow \mathbb{R}$
are not yet specified}
\begin{align*}
&B:=|\p\Psi|^2=g^{i\bar j}\Psi_i\Psi_{\bar j}\, ,\quad
D:=|\p\Psi|_{g'}^2=g'^{k\bar l}\Psi_k\Psi_{\bar l},\\
&Z:=\log B-\gamma(\Psi),\quad K:=Z-\sup_{\mathfrak D} Z+\eps S\; .
\end{align*}
Consider $\kappa$ and $S=||s||^{2\kappa}$ as in Lemma \ref{auxiliary function}. Recall that $0<\a<\mu=\b^{-1}-1$.
\begin{lem}\label{lem:4_15}
Suppose that $\Psi\in C^{2,\a}_\b$ with $\a >0$ and $\b_i <{1\over 2}$, $\forall 1\leq i \leq k$. Then for any $\kappa$ satisfies $\b_i \leq 2\kappa < (1+ \a)\b_i$, $\forall 1\leq i \leq k$, the function $K = Z-\sup_{\mathfrak D} Z+\eps S$
achieves its maximum away from $\mathfrak D$
and $|\p S|_{\Om}\leq C$.
\end{lem}
\begin{proof}The second {claim} follows directly from the formula
{(cf. Lemma \ref{auxiliary function})}
$$g_\Om^{1\bar 1}\frac{\p}{\p z^1} S\frac{\p}{\p z^{\bar 1}} S =O( |z^1|^{2(1-\b)+4\kappa-2})$$
and the fact that the exponent is non-negative. 

Now we verify
the first statement. We only concern one direction $\frac{\p}{\p z^1}$ perpendicular to one component of the divisor defined by $z^1=0$, other directions are verified similarly. We have $$
\p_1 Z=B^{-1}g^{i\bar j}(\nabla_{1}\nabla_i\Psi\cdot\Psi_{\bar j}+\Psi_{i}\cdot\p_{1}\p_{\bar j}\Psi)-\gamma' \Psi_1\; .
$$
In order to prove $Z \in C^{1,\a}_\b$, it suffices to prove that $\nabla_{1}\nabla_i\Psi\in C^{\a}_\b$ which follows from Remark \ref{pure}. On the other hand, $|\p_1^{1+\a} S|=O(|z^1|^{2\kappa-(1+\a)\b})$ with negative power. Thus we see that $S$ grows extremely faster than $Z-\sup_{\mathfrak D} Z$ near the divisor. Since $Z-\sup_{\mathfrak D} Z$ is non-positive on ${\mathfrak D}$ while $S$ vanishes along ${\mathfrak D}$, we obtain that the maximum point of $Z-\sup_{\mathfrak D} Z+\eps S$ must be achieved on $\mathfrak M$.
\end{proof}

With the lemmas above, we could assume that $p$ in the interior of $\mathfrak{M}$ is the maximum point of $K$
and choose the normal coordinate around $p$.
We {get} at $p$,
\begin{align*}
g_{i\bar j}=\delta_{ij}\text{ and }\frac{\p g_{i\bar j}}{\p z^k}
=\frac{\p g_{i\bar j}}{\p z^{\bar j}}=0\; ;
\end{align*} so
\begin{align*}
\Psi_{i\bar j}=\Psi_{i\bar i}\delta_{ij}\text{ and }g'^{i\bar j}=\frac{\delta_{ij}}{1+\Psi_{i\bar i}}\;.
\end{align*}
We have
\begin{align*}
\tri' K
&=B^{-1}\sum_{k,i,j}\frac{1}{1+\Psi_{k\bar k}}[R_{i\bar j k\bar k}\Psi_j\Psi_{\bar i}
+\Psi_{k\bar k i}\Psi_{\bar i}
+\Psi_i\Psi_{k\bar k\bar i }
+\Psi_{ik}\Psi_{\bar i\bar k}+\Psi_{i\bar k}\Psi_{\bar i k}] 
\\
&-\gamma'\tri'\Psi
-\gamma''D
-B^{-2}g'^{k\bar l}B_kB_{\bar l}
+\eps \tri'  S\;.
\end{align*}
We deal with these terms by means of the next  lemmas.
\begin{lem}
The following inequality holds
\begin{align*}
B^{-1}\sum_{k,i,j}\frac{1}{1+\Psi_{k\bar k}}R_{i\bar j k\bar k}\Psi_i\Psi_{\bar j}
-\gamma'\tri'\Psi
&\geq \sum_{k,i,j}\frac{1}{1+\Psi_{k\bar k}}(\inf R_{i\bar j k\bar k}+\gamma')
-(n+1)\gamma' \;.
\end{align*}
\end{lem}
\begin{proof}
From
$$
\tri'\Psi=\sum_k\frac{\Psi_{k\bar k}}{1+\Psi_{k\bar k}}=n+1-\sum_k\frac{1}{1+\Psi_{k\bar k}}\;,
$$
we have the lemma.
\end{proof}
\begin{lem}
The following formula holds
\begin{align*}
B^{-1}\sum_{k,i}\frac{1}{1+\Psi_{k\bar k}}\Psi_i\Psi_{ k\bar k\bar i}=
B^{-1}\sum_{k,i}\frac{1}{1+\Psi_{k\bar k}}\Psi_{\bar i}\Psi_{ k\bar k i}
\leq 1+|\p\log f|_\Om+|\p\Psi_1|_\Om\;.
\end{align*}
\end{lem}
\begin{proof}
Differentiating the equation {\eqref{eq 0 interior laplacian estimate}}, we have
\begin{align*}
g'^{i\bar j}(\p_kg_{ i\bar j}+\Psi_{i\bar j k})-g^{i\bar j}\p_kg_{i\bar j}=\p_k F \;,
\end{align*}
or
\begin{align}\label{gradient equ d}
\sum_i\frac{\Psi_{i\bar i k}}{1+\Psi_{i\bar i}}=F_k=\p_k\Psi+\p_k(\log f-\Psi_1)\;.
\end{align}
Then \eqref{gradient equ d} implies
\begin{align*}
B^{-1}\sum_{k,i}\frac{1}{1+\Psi_{k\bar k}}\Psi_i\Psi_{ k\bar k\bar i}
&=B^{-1}\sum_{k,i}\frac{1}{1+\Psi_{k\bar k}}\Psi_{\bar i}\Psi_{ k\bar k i}\\
&=B^{-1}\Psi_i F_{\bar i}\\
&=1+B^{-1}\Psi_i \tilde F_{\bar i}\\
&\leq 1+|\p\tilde F|_\Om\;.
\end{align*}
Here $\tilde F=\log f-\Psi_1$.
\end{proof}
\begin{lem}
The following formula holds
\begin{align*}
-B^{-2}g'^{k\bar l}B_kB_{\bar l}+B^{-1}\sum_{k,i,j}\frac{1}{1+\Psi_{k\bar k}}[
\Psi_{ik}\Psi_{\bar i\bar k}+\Psi_{k\bar k}\Psi_{\bar k k}]
\geq -(\gamma'+\eps) -\eps B^{-1}|\p S|_\Om\;.
\end{align*}
\end{lem}
\begin{proof}
At $p$ we have
$$
0=(Z_k+\eps S_k)(p)=B^{-1}B_k-\gamma'\Psi_k+\eps S_k \;'
$$
i.e.
$$
B^{-1}B_k=\gamma'\Psi_k-\eps S_k \;.
$$
Also,
$$
0=(Z_{\bar l}+\eps S_{\bar l})(p)=B^{-1}B_{\bar l}-\gamma'\Psi_{\bar l}+\eps S_{\bar l}\; ,
$$
i.e.
$$
B^{-1}B_{\bar l}=\gamma'\Psi_{\bar l}-\eps S_{\bar l}\;.
$$
Since at $p$ we have
$$
B_k=\Psi_{ik}\Psi_{\bar i}+\Psi_{i}\Psi_{\bar i k}\; ,
$$
and
$$
B_{\bar l}=\Psi_{i\bar l}\Psi_{\bar i}+\Psi_{i}\Psi_{\bar i \bar l}\; ,
$$
we obtain
\begin{align*}
\Psi_{ik}\Psi_{\bar i}=B(\gamma'\Psi_k-\eps S_k)-\Psi_{i}\Psi_{\bar i k}
\end{align*}
and
\begin{align*}
\Psi_{i}\Psi_{\bar i \bar l}
=B(\gamma'\Psi_{\bar l}-\eps S_{\bar l})-\Psi_{i\bar l}\Psi_{\bar i}\; .
\end{align*}
So
\begin{align*}
g'^{k\bar l}B_kB_{\bar l}
&=g'^{k\bar l}[\Psi_{ik}\Psi_{\bar i}\Psi_{j\bar l}\Psi_{\bar j}
+\Psi_{i}\Psi_{\bar i k}\Psi_{j}\Psi_{\bar j \bar l}
+\Psi_{ik}\Psi_{\bar i}\Psi_{j}\Psi_{\bar j \bar l}
+\Psi_{i\bar l}\Psi_{\bar i}\Psi_{j}\Psi_{\bar j k}]
\\
&=g'^{k\bar l}\{B(\gamma'\Psi_k-\eps S_k)\Psi_{j\bar l}\Psi_{\bar j}
-\Psi_{i\bar l}\Psi_{\bar i}\Psi_{j}\Psi_{\bar j k}
+\Psi_{ik}\Psi_{\bar i}\Psi_{j}\Psi_{\bar j \bar l}\}\; '
\end{align*}
using the normal coordinate at $p$ and assuming $\gamma'>0$ we have
\begin{align*}
g'^{k\bar l}B_kB_{\bar l}
&=\sum_{k,i,j}\frac{1}{1+\Psi_{k\bar k}}\{B(\gamma'\Psi_k-\eps S_k)\Psi_{k\bar k}\Psi_{\bar k}
-(\Psi_{k\bar k}\Psi_{\bar k})^2
+\Psi_{ik}\Psi_{\bar i}\Psi_{j}\Psi_{\bar j \bar k}\}\\
&\leq B\sum_{k,i,j}\frac{1}{1+\Psi_{k\bar k}}\{(\gamma'\Psi_k-\eps S_k)
(\Psi_{k\bar k}+1)\Psi_{\bar k}
-(\Psi_{k\bar k})^2
+\Psi_{ik}^2\}\\
&\leq (\gamma'+\eps) B^2+\eps B|\p S|_\Om+B\sum_{k,i}\frac{1}{1+\Psi_{k\bar k}}\{-(\Psi_{k\bar k})^2
+\Psi_{ik}^2\}\; .
\end{align*}
So
\begin{align*}
&-B^{-2}g'^{k\bar l}B_kB_{\bar l}+B^{-1}\sum_{k,i,j}\frac{1}{1+\Psi_{k\bar k}}[
\Psi_{ik}\Psi_{\bar i\bar k}+\Psi_{k\bar k}\Psi_{\bar k k}]\\
&\geq -(\gamma'+\eps) -\eps B^{-1}|\p S|_\Om+2B^{-1}\sum_{k}\frac{(\Psi_{k\bar k})^2}{1+\Psi_{k\bar k}}\;.
\end{align*}
\end{proof}

\begin{prop}\label{prop: gradient estimate}
We have the gradient estimate
\begin{align*}
\sup_{\mathfrak{X}}|\p\Psi|^2_{\Om}\leq C_{i} \text{ for } i=1,2 \;.
\end{align*}
The constants $C_i$ depend on, respectively,
\begin{align*}
C_1&=C_1(\inf_{i\neq k}R_{i\bar i k\bar k}(\Om),\sup|\p\log f|_\Om,\sup|\p \Psi_1|_\Om,
\osc_{\mathfrak X}\Psi,\osc_{\p\mathfrak X}\Psi_1)\, ,\\
C_2&=C_2(\inf_{i\neq k}R_{i\bar i k\bar k}(\Om_1),\sup|\p \Psi_1|_\Om,\sup\tr_{\Om}\Om_1,
\sup\tr_{\Om_1}\Om,\osc_{\mathfrak X}\Psi,\osc_{\mathfrak X}\Psi_1) \; .
\end{align*}
\end{prop}
\begin{proof}
We assume $B(p)\geq 1$, otherwise we are done.
We compute
\begin{align*}
 \Delta ' (Z-\sup_{\mathfrak D} Z+\eps S)&\geq \sum_{i,j,k} \frac{1}{1+\Psi_{k \bar k}}
\left(
\inf_{i\neq k} R_{i\bar i k\bar k} + \gamma '
\right) -(n+1)\gamma ' - \gamma '' D -1 \\
& -|\partial \log f|_\Om - |\partial \Psi_1|_\Om - |\gamma ' +\eps|
-\eps B^{-1} |\partial S|_\Om - C\eps \tr_{\Omega '}\Omega \\
& = \left(
\inf_{i\neq k} R_{i\bar i k\bar k} + \gamma ' -C\eps
\right) \tr_{\Omega ' }\Omega
-\gamma '' D -1 -|\partial \log f|_\Om - |\partial \Psi_1|_\Om\\
&-(n+1)\gamma ' -(n+2)\gamma ' - \eps -\eps |\partial S|_\Om \; .
\end{align*}
We choose an appropriate $\gamma$, for example
\begin{align*}
 \gamma (t):= -C' e^{\sup \Psi - t},\,
\mbox{ where }  C':= \inf_{i\neq k} R_{i\bar i k \bar k} -C\eps +1 \; .
\end{align*}
We notice that $|\nabla S|$ is bounded by means of Lemma \ref{lem:4_15}.
Then,
\begin{align*}
 D + \tr_{\Omega '}\Omega (p)\leq C
= C(\inf_{i \neq  k} R_{i\bar i k \bar k},\, \sup|\p \log f|_\Om ,\, \sup |\p \Psi_1 |_\Om)\; .
\end{align*}

Since
\begin{align*}
 \left(
\tr_{\Omega}\Omega '
\right)^{\frac{1}{n}}
=
\left(
n+1 +\Delta \Psi
\right)^{\frac{1}{n}}
\leq
\tr_{\Omega '}\Omega \cdot e^{\frac{F}{n}}\; ,
\end{align*}
so $B\leq \tr_{\Omega}\Omega ' \cdot D\leq C$.
Moreover, for any $x\in \Int(\mathfrak{X})$, there holds
\begin{align*}
 \log B(x)
&= K(x)+ \gamma(\Psi)(x)-\eps S(x)+\sup_{\mathfrak D} Z \\
&\leq K(p) + \sup_{\partial \mathfrak{X}}K+ \gamma(\Psi)(x)-\eps S(x)+\sup_{\mathfrak D} Z \\
&= \log B(p) -\gamma(\Psi)(p)-\sup_{\mathfrak D} Z+\eps S(p)  + \sup_{\partial \mathfrak{X}}K+ \gamma(\Psi)(x)-\eps S(x)+\sup_{\mathfrak D} Z\\
&\leq \log B(p) - \gamma (\Psi) (p) + \gamma(\Psi)(x)
    + \sup_{\p \mathfrak{X}}(\log B -\sup_{\mathfrak D} \log B- \gamma(\Psi))+C\;.
\end{align*}
Here we use the assumption that $B\geq1$, so $\log B\geq 0$.
Similarly to former arguments, we change the background metric and we consider
\begin{align*}
 \log \frac{\det (\Omega_{1,i\bar j } + \tilde{\Psi}_{i\bar j})}{\det (\Omega_{1, i\bar j})}
=F_1 = \log \tau + \tilde{\Psi}\; .
\end{align*}
We arrive at
\begin{align*}
 \sup_{\mathfrak{X}} |\partial \tilde{\Psi}|_{\Omega_1}^2 \leq C \;.
\end{align*}
As a result, the proof of the proposition follows from
\begin{align*}
 \sup_{\mathfrak{X}} |\partial \Psi |_{\Omega}
\leq \sup_{\mathfrak{X}}
\tr_{\Omega}\Omega_1 \cdot ( \sup_{\mathfrak{X}} |\partial \tilde{\Psi}|_{\Omega_1} + \sup_{\mathfrak{X}} |\partial\Psi_1|_{\Omega_1})\;.
\end{align*}
\end{proof}


\section{Solving the geodesic equation}\label{open}
In this section, we assume that the components of $D$ are smooth and disjoint.


\subsection{Existence of the $C^{1,1}_\b$ cone geodesic}\label{Existence of the weak geodesic}
 In the present subsection we are dealing with the Dirichlet problem
for the family of approximate geodesic equation \eqref{per equ sim}.
In order to apply the a priori estimates in Section \ref{close},
we require that the pair $(\Om,\Om_1)$ satisfies are $|\p\log\tilde\Om_1^{n+1}|$, $|\p\log\frac{\Om_1^{n+1}}{\Om^{n+1}}|$ bounded and one of the following conditions
\begin{itemize}
  \item $\vert Riem(\Om_1)\vert$ is bounded;
  \item $\inf Riem(\Om_1)$ and $\sup Riem(\Om)$ are bounded;
  \item $\sup \Ric(\Om_1)$ and $\inf Riem(\Om)$ are bounded;
  \item $\inf \Ric(\Om_1)$ and $\vert Riem(\Om)\vert$ are bounded.
\end{itemize}
Then we reduce these conditions to
geometric conditions on the boundary potentials $\vphi_0$ and $\vphi_1$ as follows.

The boundedness
of the connection of the background cone metric $\om$ in \eqref{model cone}
is computed in the following lemma for $0<\b_1<\frac{2}{3}$.
It was also computed for $0<\b_1<\frac{1}{2}$ in Brendle \cite{Brendle:2011kx}.
\begin{lem}\label{lemma: connection bounded}
The connection of $\om$ is bounded for $0<\b_1<\frac{2}{3}$ under the coordinate chart $\{w^i\}$.
\end{lem}
\begin{proof}
Since there exists a smooth function $\rho$ such that
$\delta |s|^{2\b_1}_{h_\L} = \rho |z^1|^{2\b_1}$, we can rewrite \eqref{model cone} as
\begin{align*}
\om&=\om_0
+\frac{\sqrt{-1}}{2}|z^1|^{2\b_1}\rho_{k\bar l} dz^k\wedge dz^{\bar l}\\
&+\frac{\sqrt{-1}}{2} \b_1|z^1|^{2(\b_1-1)} (z^1\rho_{k} dz^k\wedge dz^{\bar 1}
+z^{\bar 1} \rho_{\bar l} dz^1 \wedge dz^{\bar l})\\
&+\frac{\sqrt{-1}}{2} \b_1^2 \rho|z^1|^{2(\b_1-1)} dz^1\wedge dz^{\bar 1}
\end{align*}for $k,l$ from $2$ to $n$. By means of the
change of coordinates \eqref{conetransform}, as $w^i=|z^i|^{\b_1-1}z^i$,
we have, for $i\in \{1, \cdots , n\}$
\begin{align*}
\frac{\p w^i}{\p z^i}=\frac{\b_i+1}{2}|z^i|^{\b_i-1};
\frac{\p w^i}{\p z^{\bar i}}=\frac{\b_i-1}{2}|z^i|^{\b_i-3}z^iz^i.
\end{align*}
Meanwhile,
\begin{align*}
\frac{\p z^i}{\p w^i}=\frac{1+\b_i}{2\b_i}| w^i|^{\frac{1-\b_i}{\b_i}};
\frac{\p z^i}{\p w^{\bar i}}=\frac{1-\b_i}{2\b_i}| w^i|^{\frac{1-3\b_i}{\b_i}}w^iw^i.
\end{align*}
The components of the model cone metrics under the variables $w^i$ rbecome
\begin{align*}
\tilde g_{1\bar 1}
&=[(\frac{1+\b_1}{2\b_1})^2| w^1|^{\frac{2}{\b_1}-2}+(\frac{1-\b_1}{2\b_1})^2| w^1|^{\frac{2}{\b_1}-2}] g_{1\bar 1}\circ W^{-1}\\
&=\frac{1+\b_1^2}{2\b_1^2}| w^1|^{\frac{2}{\b_1}-2} [g_{01\bar 1}\circ W^{-1}+| w^1|^2\rho_{1\bar 1}\\
&+\b_1|w^1|^{\frac{2}{\b_1}-2}
(|w^1|^{\frac{1}{\b_1}-1}w^1\rho_1
+|w^1|^{\frac{1}{\b_1}-1}w^{\bar 1}\rho_{\bar 1})
+\b_1^2\rho|w^1|^{2-\frac{2}{\b_1}}]\\
&=\frac{1+\b_1^2}{2\b_1^2}[g_{01\bar 1}\circ W^{-1}| w^1|^{\frac{2}{\b_1}-2}
+|w^1|^{\frac{2}{\b_1}}\rho_{1\bar 1}
+\b_1(w^1\rho_1+w^{\bar 1}\rho_{\bar 1}+\b_1^2\rho)],\\
\tilde g_{1\bar l}
&=\frac{1+\b_1}{2\b_1}
[|w^1|^{\frac{1}{\b_1}-1}g_{01\bar l}\circ W^{-1}
+|w^1|^{\frac{1}{\b_1}+1}\rho_{1\bar l}\circ W^{-1}
+\b_1w^1\rho_{\bar l}\circ W^{-1}],\\
\tilde g_{k\bar l}&=g_{0k\bar l}\circ W^{-1}
+|w^1|^2\rho_{k\bar l}\circ W^{-1}.
\end{align*}
Now, the connection of $\om$ is the first derivative with respect to $w^i$.
We check one by one. Note that $\rho$ is smooth on $w^k$ for $1\leq k\leq n$.
\begin{align*}
&\frac{\p}{\p w^1}\tilde g_{1\bar 1}
=O(| w^1|^{\frac{2}{\b_1}-3}+| w^1|^{\frac{2}{\b_1}-1});\\
&\frac{\p}{\p w^i}\tilde g_{1\bar 1}
=O(1);\\
&\frac{\p}{\p w^i}\tilde g_{1\bar l}
=O(1);\\
&\frac{\p}{\p w^1}\tilde g_{k\bar l}=\frac{\p}{\p w^i}\tilde g_{k\bar l}
=O(1).
\end{align*}
Now let us check
$\frac{\p}{\p w^1}\tilde g_{1\bar l}.$
It contains three terms. The first term is
\begin{align*}
&\frac{\p}{\p w^1}(|w^1|^{\frac{1}{\b_1}-1}g_{01\bar l}\circ W^{-1})\\
&=\frac{\p}{\p w^1}[(|w^1|^{\frac{1}{\b_1}-1}|w^1|)(|w^1|^{-1}g_{01\bar l}\circ W^{-1})].
\end{align*}
Since $g_{01\bar l}\circ W^{-1}$ is also smooth and converges to zero as $w^1$ goes to zero,
so this first term is $O(|w^1|^{\frac{1}{\b_1}-1})$.
The second and third term are both $O(1)$.
Thus we conclude that when $0<\b_1<\frac{2}{3}$, the connection is bounded.
\end{proof}
As a corollary, we arrive at the boundedness of the connection of $\Om_1$.
\begin{cor}\label{connection}
When $0<\b_1<\frac{2}{3}$ and $\vphi_0,\vphi_1\in C_\b^{3}$,
the connection of $\Om_1$ is bounded.
\end{cor}
\begin{proof}
From Lemma \ref{lemma: connection bounded} and the expression of $\Om$ in \eqref{Om},
we know the connection of $\Om$ is bounded for $0<\b_1<\frac{2}{3}$. Recall the formula \eqref{Om1} of $\Om_1$; we have
\begin{align*}
\Om_1
&=t\om_{\phi}+(1-t)\om_{\vphi}
+\frac{\sqrt{-1}}{2}(1+m\p_{n+1}\p_{\overline{n+1}}\Phi)dz^{n+1} \wedge d\bar z^{n+1}\\
&+\frac{1}{\sqrt{2}}\p_i(\phi-\vphi)dz^idz^{\overline{n+1}}
+\frac{1}{\sqrt{2}}\p_{\bar i}(\phi-\vphi)dz^{\bar i}dz^{n+1}.
\end{align*}
We have the the component of $\Om_1$ to be for $2\leq i,j\leq n$,
\begin{align*}
&( g_1)_{1\bar 1} = t (g_{\vphi_0})_{1\bar 1}+(1-t) (g_{\vphi_1})_{1\bar 1};\\
&(g_1)_{1\bar i} =t (g_{\vphi_0})_{1\bar i}+(1-t) ( g_{\vphi_1})_{1\bar i};\\
&( g_1)_{1\overline {n+1}}=\p_1(\vphi_0- \vphi_1);\\
&( g_1)_{i\bar {j}}=t (g_{\vphi_0})_{i\bar j}+(1-t) (g_{\vphi_1})_{i\bar j};\\
&( g_1)_{i\overline {n+1}}=\p_i(\vphi_0-\vphi_1);\\
&( g_1)_{n+1\overline {n+1}}=1+m\p_{n+1}\p_{\overline{n+1}}\Phi.
\end{align*}
Thus the corollary follows from $\vphi_0,\vphi_1\in C_\b^{3}$.
\end{proof}

\begin{lem}\label{lemma: curvature line}
Suppose that $\vphi_0, \vphi_1 \in C^{3}_\b$ have curvature lower (upper) bound. Then $\Om_1$ has also curvature lower (resp. upper) bound.
\end{lem}
\begin{proof}
Since the formula of the bisectional curvature is
\begin{align*}
R_{i\bar{j}k\bar{l}}=-\frac{\partial^2g_{i\bar{j}}}{\partial z^k\partial
z^{\bar{l}}}
+g^{p\bar{q}}\frac{\partial g_{p\bar{j}}}{\partial
z^{\bar{l}}}\frac{\partial g_{i\bar{q}}}{\partial z^{k}},
\end{align*}
we have for $1\leq i,j,k, l\leq n$ and $\phi=t\vphi_1+(1-t)\vphi_0$,
\begin{align*}
R_{i\bar{j}k\bar{l}}(g_1)
&=tR_{i\bar j k\bar l}(g(\vphi_1))+(1-t) R_{i\bar j k\bar l}(g(\vphi_0)) \\
&-tg(\vphi_1)^{p\bar{q}} \p_{\bar{l}} g(\vphi_1)_{p\bar{j}}\p_{k} g(\vphi_1)_{i\bar{q}}\\
&-(1-t)g(\vphi_0)^{p\bar{q}}  \p_{\bar{l}} g(\vphi_0)_{p\bar{j}}   \p_{k}g(\vphi_0)_{i\bar{q}}\\
&+\sum_{1\leq p, q\leq n} g_1^{p\bar q} \p_{\bar l} g(\phi)_{p\bar j} \p_{k} g(\phi)_{i\bar q}\\
&+\sum_{1\leq p \leq n} g_1^{p \overline{n+1}} \p_{\bar l} g(\phi)_{p\bar j}
\frac{1}{\sqrt 2}\p_k\p_i[\vphi_1-\vphi_0]\\
&+\sum_{1\leq q \leq n} g_1^{n+1\bar q}
\p_{k} g(\phi)_{i\bar q}
\frac{1}{\sqrt 2}\p_{\bar l}\p_{\bar j}[\vphi_1-\vphi_0]\\
&+ \frac{1}{2}g_1^{n+1\overline{n+1}}
\p_k\p_i[\vphi_1-\vphi_0]
\p_{\bar l}\p_{\bar j}[\vphi_1-\vphi_0].
\end{align*}
Also,
\begin{align*}
R_{i\bar{j} k \overline{n+1}}(g_1)
&=-\p_k (\vphi_1-\vphi_0)_{i\bar{j}} \\
&+\sum_{1\leq p,q\leq n}g_1^{p\bar q}(\vphi_1-\vphi_0)_{p\bar{j}} \p_k g( \phi)_{i\bar q}\\
&+g_1^{p\overline{n+1}}(\vphi_1-\vphi_0)_{p\bar{j}} \frac{1}{\sqrt 2}\p_k\p_i[\vphi_1-\vphi_0];\\
R_{i\bar{j} (n+1) \overline{n+1}} (g_1)
&=\sum_{1\leq p, q\leq n}
m g_1^{p\bar{q}}  (\vphi_1-\vphi_0)_{p\bar{j}}  (\vphi_1-\vphi_0)_{i\bar{q}};\\
R_{i \overline{n+1} (n+1) \overline{n+1}} (g_1)
&=\sum_{q=1}^n m g_1^{n+1 \bar q} \bar\p_{n+1} \p_{n+1} \bar\p_{n+1} \Psi (\vphi_1-\vphi_0)_{i\bar q};\\
R_{n+1\overline{n+1} (n+1) \overline{n+1}} (g_1)
&=-m\p_{n+1}\bar\p_{n+1}\p_{n+1}\bar\p_{n+1}\Psi\\
&+m^2 g_1^{n+1 \overline{n+1} } \bar\p_{n+1}\bar\p_{n+1}\p_{n+1} \Psi
\p_{n+1}\bar\p_{n+1}\p_{n+1} \Psi.
\end{align*}
The connection and the lower bound of the curvature of $\vphi_0$ and $\vphi_1$ are bounded. So the curvature of $\Om_1$ is also bounded below. The upper bound follows in the same way.
\end{proof}
\begin{cor}\label{cor: curvature line}
Suppose that $0<\b_1 < \frac{2}{3}$, $\vphi_0, \vphi_1 \in C^{3}_\b$ and their Ricci curvature have lower (upper) bound. Then the Ricci curvature of $\Om_1$ also has lower (resp. upper) bound.
\end{cor}
\begin{proof}
We use the formulas of the Riemannian curvature in Lemma \ref{lemma: curvature line}, and we take the trace
to obtain the Ricci curvature. Then the lemma follows directly.
\end{proof}

Since $\inf Riem(\Om_1)$ is bounded for $0<\b_1<\frac{1}{2}$ and $\sup Riem(\Om_1)$
is bounded for $0<\b_1<1$ (c.f. \cite{Brendle:2011kx}\cite{Jeffres:2011vn}),
we introduce the following subspaces of K\"ahler cone metrics.
When $0<\b_1<\frac{1}{2}$, we define
\begin{align*}
\mathfrak{I_1}&:=\{\vphi\in\mathcal H^{3}_\b\vert  \sup \Ric(\om_\vphi)
\text{ is bounded}\};
\end{align*}
\begin{align*}
\mathfrak{I_2}&:=\{\vphi\in\mathcal H^{3}_\b\vert  \inf \Ric(\om_\vphi)
\text{ is bounded}\};
\end{align*}
\begin{thm}\label{sole geo}
Assume that two K\"ahler cone potentials $\vphi_0, \vphi_1$ are both in $\mathfrak{I_i}$ $i=1,2$.
Then they are connected by a $C^{1,1}_\b$ cone geodesic.
\end{thm}
\begin{proof}
Note that the right hand side of the equation is positive as long as $\tau$ is positive.
When $\tau$ is zero, \eqref{per equ} provides a solution of
the geodesic equation \eqref{geo ma}.


We denote the set of solvable times of \eqref{per equ} by
$$
I=\{\tau\in(0,1]|\eqref{per equ}_\tau \text{ is solvable in } C^{2,\a}_\b\}\;.
$$
Automatically, $\Psi=\Psi_1$ satisfies the equation at $\tau=1$ , so the set $I$ is not empty.

For any $0< \tau\leq 1$, assuming that $\om(\tau_0)$ solves the equation \eqref{per equ},
Proposition \ref{dp} provides a unique solution in $C^{2,\a}_\b$
to the following linearized equation
\begin{equation*}
  \left\{
   \begin{array}{ll}
\tri_{\tau_0} v-v=f &\text{ in }\mathfrak{M}\;,\\
v=u &\text{ on }\p\mathfrak{X}\;,
   \end{array}
  \right.
\end{equation*}
for any $f\in C^\a_\b$ and $u\in C^{2,\a}_\b$.
So the linearized operator at $\tau_0$ is invertible, and thus $I$ is open.
So the solvable time can be extended beyond $\tau_0$.

The a priori estimates in Section \ref{close},
with one of the geometry conditions in  $\mathfrak{ I_1}$ or $\mathfrak{ I_2}$ assures the uniform
$C_\b^{1,1}$ bound of $\vphi(t)$ which is independent of $\tau$.
Two estimates in the next subsections improve $C^{2,\a}_\b$ regularity of the
solution of \eqref{per equ} before $\tau=0$. Thus, we could solve th approximation equation till $\tau=0$. With the uniform $C_\b^{1,1}$ bound,
after taking a subsequence $t_i$ we have a weak limit $\vphi=\lim_{t_i\rightarrow0}\vphi(t_i)$
under a $C_\b^{1,\a}$ norm. In Section \ref{unigeo}, we prove the uniqueness of a weak solution.
Hence the theorem is proved completely.
\end{proof}

\subsection{Interior Schauder estimate: $\tau>0$}
We first prove the $C^{2,\a}_\b$ estimate for a general equation.
\begin{align}\label{eq_starting_@_interior_holder_estimate}
 \log \Omega_{\Psi}^{n+1} = \log \Omega^{n+1} + F. \;
\end{align}
\begin{prop}Assume that we have the second order estimate of $\Psi$.
Then the following estimate holds for the solution of \eqref{eq_starting_@_interior_holder_estimate} on any small ball $B\subset \mathfrak X$
\begin{align}\label{claimed_est_2_alpha_interior}
 |\sqrt{-1}\p\bar\p \Psi|_{C^\a_\b(B)} \leq C \;,
\end{align}
where $C$ depends on
$|\p\log\tilde\Om^{n+1}|_{L^q}$,
$|\log\Om^{n+1}|_{C^\a_\b}$,
$|\p\Psi|_{\infty}$,$|\tri\Psi|_{\infty} $,
$|\p \tilde F|_{L^q}$, $|F|_{C^\a_\b}$,
where $q > 2n+2 $.
\end{prop}
\begin{proof}Choose a small ball $B_d(p)$ around $p$ in the interior of $\mathfrak X$. When $B_d(p)$ does not intersect $\mathfrak D$, this proposition follows directly from the standard Evans-Krylov estimate. So it's sufficient to fix a point $p \in \mathfrak{D}$.
We consider \eqref{eq_starting_@_interior_holder_estimate} in $B_d (p)$.
We consider it in a local holomorphic coordinate chart and we differentiate it in $B_d(p)\setminus\mathfrak D$.
We fix the following piece of notation
$$h : =  \log \Omega^{n+1}+F.$$
So, we fix a $1\leq k \leq n+1$ and,  by taking $\frac{\partial}{\partial z^k}$
on both sides
of \eqref{eq_starting_@_interior_holder_estimate} we get
\begin{align*}
 {g'}^{i\bar j} (g_{i\bar j k} + \Psi_{i \bar j k} )
 = h_k \;.
\end{align*}
Taking $\frac{\partial}{\partial z^{ \bar l}}$ on both sides of the above equation,
we have
\begin{align*}
 &-{g'}^{p\bar j} {g'}^{i\bar q} (g_{p\bar q \bar l} + \Psi_{p\bar q \bar l})
(g_{i\bar j k} + \Psi_{i\bar j k}) +
{g'}^{i\bar j}(g_{i\bar j k \bar l} + \Psi_{i\bar j k \bar l} )
=
h_{k\bar l}
\;.
\end{align*}
We introduce the notation $V:=g+ \Psi $.
Here $ g$ is the local potential of $g_{i\bar j}$ in $B_d(p)$,
then
\begin{align*}
{g'}^{i\bar j} V_{k\bar l i \bar j}
&={g'}^{p\bar j}{g'}^{i\bar q}V_{p\bar q k}V_{i\bar j \bar l}
+h_{k\bar l}  \;  .\\
\;
\end{align*}
Since this equation is not well-defined along $\mathfrak D$,
we choose inverse of the flat metric $g^{k \bar l}$ \eqref{flat cone} as the weighted function $\sigma^{k \bar l}$, and
we consider $ {g'}^{i \bar j}(\sigma^{k \bar l} V_{k \bar l})_{i \bar j} $.
We have
\begin{align*}
\tri'(g^{k \bar l} V_{k \bar l}) =
{g'}^{i \bar j}(
R^{k \bar l}_{\phantom{k \bar l}i \bar j } V_{k \bar l }
+ g^{k \bar l} V_{k \bar l i \bar j})\; .
\end{align*}
Now given any direction $\eta\in \mathbb{C}^{n+1}$, with $|\eta|= 1$,
we denote
$\p_\eta :=\sum_k \eta^k\frac{\p}{\p z^k} $.
Also, we set
$V_{\eta\bar\eta} := \p^2_{\eta \bar\eta} V=\sum_{k,l}
\eta^k \eta^{\bar l}\frac{\p}{\p z^k}\frac{\p}{\p z^{\bar l}}V$.
We then define
$u_\eta := \sum_{k,l} \eta^k \eta^{\bar l} \sigma^{k\bar l } V_{k \bar l} $.
We have
\begin{align*}
\tri'(u_\eta)
&\geq \sum_{k,l} \eta^k \eta^{\bar l}
\sigma^{k \bar l}h_{k\bar l}\; ,
\end{align*}
where we use that the flat cone metric has flat curvature under the coordinate ${w^i}$.

We denote
$\tilde h_{\bar l}:= g^{k \bar l}\tilde h_{k}$ on the coordinate chat $\{w^i\}$.
Let us now introduce the following symbols.
We denote
$$
M_{s\eta} := \sup_{B_{sd}(p)} u_\eta,\quad
 m_{s\eta} := \inf_{B_{sd}(p)} u_\eta \; .$$
Applying
Proposition \ref{harnack} (weak Hanack inequality) to $M_{2\eta} - u_\eta$,
we have that there exists a $q>2n+2$ such that
\begin{align}\label{harnack_@_interior_holder_estimate}
& \left\{
d^{-2n-2} \int_{B_d (p)}  (M_{2\eta} - u_\eta)^p \Omega^{n+1}
\right\}^{\frac{1}{p}}
\leq
C
 \left\{
 M_{2\eta} - M_\eta + K
\right\}\;.
\end{align}
Here $$K:=d^{1-\frac{2n+2}{q}} \|\p{\tilde h}\|_{q}\; $$
and
\begin{align*}
\|\p{\tilde h}\|_{q}= \|{\p\log\tilde\Om^{n+1}+\p\tilde F}\|_{q}.
\end{align*}

In order to obtain the inverse inequality for $u_\eta - m_{2\eta}$
we use the concavity of the Monge-Amp\`ere operator.
Fix any two points $Q_2 \in B_{2d}(p)$ and $Q_1 \in B_{d}(p)$, without loss of generality,
we assume the distance from $Q_2$ to $\mathfrak D$ is longer than $Q_1$ to $\mathfrak D$.
From the formula of the flat metric \eqref{flat cone}, we see that $g^{i \bar j }(Q_2)>g^{i \bar j }(Q_1)$.
From the equation \eqref{eq_starting_@_interior_holder_estimate}
we have, writing
${g'}(t) := (1-t) {g'}(Q_2) + t {g'}(Q_1)$, and $a^{i \bar j} = \int_0^1 {g'}^{i\bar j}(t) dt$,
the expression
\begin{align}\label{eq_evolved_@_interior_holder_estimate}
 &h(Q_1) - h(Q_2) =
\log \det ({g'}_{i \bar j} (Q_1))
- \log \det ({g'}_{i \bar j} (Q_2)) \nonumber\\
&=\int_0^1 {g'(t)}^{i\bar j} dt (V(Q_1) - V(Q_2))_{i \bar j}
= a^{i\bar j } (V(Q_1) - V(Q_2))_{i \bar j}\;.
\end{align}
Now, for $1\leq i ,j \leq n+1$ we define
\begin{align*}
\tilde a^{i\bar j} := \frac{a^{i\bar j}}{g^{i \bar j }(Q_2)} \; .
\end{align*}
We have (since $g'(t)$ is $L^\infty$-equivalent to $g$, for $1\leq i,j \leq n+1$)
that the matrix $\tilde{a}^{i \bar j}$ is positive definite and its
eigenvalues range between the positive constants $\lambda$ and $\Lambda$.
Thus, we can apply Lemma $17.13$ in \cite{MR1814364}
(see also Section $(4.3)$ in \cite{MR904673}); we get that
there exists a finite set of unit vectors
$\gamma_1 , \cdots , \gamma_N \in \mathbb{C}^{n+1}$
and positive numbers $\lambda^* , \Lambda^*$ depending only on $n, \lambda , \Lambda$
such that the matrix $\tilde{a}^{i \bar j}$ can be written as
\begin{align*}
 \tilde{a}^{i \bar j}
= \sum_{\nu = 1}^N b_\nu \gamma_{\nu i} \gamma_{\nu \bar j}\;.
\end{align*}
Here $\lambda^* \leq b_{\nu} \leq \Lambda^*$ for any $1\leq \nu \leq N$.
As a result, we can express the matrix $a^{i\bar j}$ in terms of $b_\nu$ and the vectors
$\gamma_\nu$.
Thus, we continue from  \eqref{eq_evolved_@_interior_holder_estimate} and we write
\begin{align*}
&h(Q_1) - h(Q_2)
=g^{i \bar j }(Q_2) \tilde a^{i\bar j } (V(Q_1) - V(Q_2))_{i \bar j}\\
&=
\sum_{\nu = 1}^N b_\nu g^{i \bar j }(Q_2) \gamma_{\nu i} \gamma_{\nu \bar j}
(V(Q_1) - V(Q_2))_{i \bar j}
\geq
C \sum_{\nu = 1}^N b_\nu   (u_{ \gamma_{\nu}} (Q_1) - u_{ \gamma_{\nu}} (Q_2))\;,
\end{align*}
where we used that the matrix $V_{i \bar j}$ is positive-definite and $g^{i \bar j }(Q_2)\geq g^{i \bar j }(Q_1)$.
We conclude that for a fixed $1\leq l \leq N$
\begin{align} \label{needed_@_interior_holder_estimate}
C b_l (u_{ \gamma_{l}} (Q_1) - u_{ \gamma_{l}} (Q_2))
\leq
h (Q_1) - h (Q_2)
 + C \sum_{\nu \neq l} b_\nu (u_{ \gamma_{\nu}} (Q_2) - u_{ \gamma_{\nu}} (Q_1))\; .
\end{align}
We now  fix $1\leq \nu \leq N$, $s=1,2$ and we denote
$$w(sd) := \sum_{\nu = 1}^N \osc_{B_{sd}(p)} u_{ \gamma_{\nu}}.$$
From \eqref{needed_@_interior_holder_estimate},
since $Q_1 \in B_{d}(p)$ and $Q_2 \in B_{2d}(p)$
we get
\begin{align*}
 u_{ \gamma_{l}} (Q_1) - m_{2l} &\leq
C
\{
d^\a |h|_{C_\b^\a} + \sum_{\nu \neq l} (M_{2\gamma_{\nu}} - u_{ \gamma_{\nu}} (Q_1))\}.
\end{align*}
Applying the inequality \eqref{harnack_@_interior_holder_estimate}, we have
\begin{align}\label{harnack_two_@_interior_holder_estimate}
 &\left\{
d^{-2n-2} \int_{B_d (p)}  ( \sum_{\nu \neq l} M_{2 \gamma_{\nu}} - u_{ \gamma_{\nu}})^p \Omega^{n+1}
\right\}^{\frac{1}{p}} \nonumber \\
&\leq
 N^{\frac{1}{p}} \sum_{\nu \neq l}
\left\{
d^{-2n-2}
\int_{B_d (p)}  (  M_{2 \gamma_{\nu}} - u_{ \gamma_{\nu}})^p \Omega^{n+1}
\right\}^{\frac{1}{p}} \nonumber \\
&\leq C
 \left\{
\sum_{\nu \neq l} (M_{2 \gamma_{\nu}}  - M_{ \gamma_{\nu}})
+ K
\right\}\nonumber\\
&\leq C
 \left\{
w(2d) - w(d)
+ K
\right\}\;
\end{align}
which entails, by integrating of $Q_1$ on $B_d (p)$ with respect to $\Om$ and using \eqref{harnack_two_@_interior_holder_estimate}
\begin{align}\label{harnack_three_@_interior_holder_estimate}
&\left\{
d^{-2n-2}\int_{B_d (p)}  (u_{ \gamma_{l}} (Q_1) - m_{2l} )^p
\Omega^{n+1}
\right\}^{\frac{1}{p}} \nonumber \\
&\qquad\leq
C
\left\{
d^\a |h|_{C_\b^\a} +
w(2d) - w(d)
+K
\right\}\;.
\end{align}
Now, we combine \eqref{harnack_two_@_interior_holder_estimate} and \eqref{harnack_three_@_interior_holder_estimate} to obtain
\begin{align*}
 w(2d)
&\leq C
\left\{
d^\a |h|_{C_\b^\a} +
w(2d) - w(d)
+ K
\right\}\;,
\end{align*}
where at the last inequality we used \eqref{harnack_@_interior_holder_estimate} and
\eqref{harnack_three_@_interior_holder_estimate}.
Let us compute that $|h|_{C_\b^\a}=|F + \log\Om^{n+1}|_{C_\b^\a}$.

Then, using the Iteration Lemma $8.23$ in \cite{MR1814364},
we have $u_\eta \in C^\a_\b$, for all $\eta\in \mathbb{C}^{n+1}$.
So $\Delta V \in C^\a_\b$ and $V\in C^{2 , \a}_\b $
follows from Proposition \ref{linear equ gl}.
This gives \eqref{claimed_est_2_alpha_interior} and completes the proof of the
proposition.
\end{proof}

In conclusion, we obtain the conical Evans-Krylov estimate of the geodesic equation \eqref{per equ}.
\begin{prop}
Assume $0<\b_1<\frac{2}{3}$ and that $\vphi_0,\vphi_1$ are in $\mathfrak{I_i}$, $i=1,2,3$.
Then the $C^{1,1}_\b$ solution $\Psi$ of the approximate geodesic equation \eqref{per equ}
belongs to $C^{2,\a}_\b$ in the interior of $\mathfrak X$.
\end{prop}
\begin{proof}
Considering the geodesic equation \eqref{per equ}, then $F=\log\tau+\log\frac{\Om_1^{n+1}}{\Om^{n+1}}+\Psi-\Psi_1$. Since $\Om\in C^{\a}_\b$, we have $\log\Om^{n+1}\in C^\a_\b$. Moreover, $\vphi_0,\vphi_1\in C^{2,\a}_\b$, so $\log\Om_1^{n+1}\in C^\a_\b$. Thus we have $F\in C^\a_\b$.
When $0<\b_1<\frac{2}{3}$, \lemref{lemma: connection bounded}, \lemref{connection} and $\vphi_0,\vphi_1\in C^{3}_\b$ imply that $\p\tilde F$ is bounded.
\end{proof}

Our argument presented above follows Evans-Krylov's estimate
\cite{MR649348}\cite{MR678347}\cite{MR688919}.
We also used Blocki's observation in \cite{MR1788046}
that $F$ belongs to $W^{1,q}$ is sufficient to the estimate.
In our problem, since $V_{k\bar l}$ is singular along
the direction which is perpendicular to $\mathfrak D$,
we multiply with the weight. In the next Section, we will develope the linear theory
including the weak Hanack inequality for the linear equation and with cone coefficient which is used in the proof above.
\subsection{An application to the K\"ahler-Einstein cone metrics}
Now we state an application of our estimate to the K\"ahler-Einstein metric on $(X,\om)$ with cone singularities. We first assume the divisor has only one component. I.e. $D=(1-\b_1)V.$ As usual, we assume that $z^1$ is the defining function of the hyper surface $V$.
The K\"ahler-Einstein cone metric satisfies for a real number $\l$, 
\begin{align*}
Ric (\omega_{\vphi})=\l\om_\vphi+2\pi [D]. 
\end{align*}
This equation implies the cohomology relation
\begin{align*}
c_1(M)=\l \Om+c_1(L_D). 
\end{align*} 
Here $L_D$ is the corresponding line bundle of $[D]$. Since $\om_0$ is a smooth K\"ahler metric in $\Om$, there exists a smooth function $f_0$ such that
\begin{align*}
Ric(\om_0)-\l\om_0+i\p\bar\p\log |s|^{2(1-\b_1)}=i\p\bar\p f_0.
\end{align*} 
Thus we have
\begin{align*}
Ric(\om)-\l\om
&=Ric(\om)-Ric(\om_0)-\l(\om-\om_0)+Ric(\om_0)-\l\om_0\\
&=i\p\bar\p f.
\end{align*} 
In which $$f=-\log(\frac{\om^n}{\om^n_0}|s|^{2(1-\b_1)})-\l\delta|s|^{2\b_1}+f_0.$$
Thus the K\"ahler-Einstein cone metric satisfies 
\begin{align}\label{keg}
 \log \omega_{\vphi}^{n} = \log \omega^{n} -\l \vphi+f=F. \;
\end{align}
When $\l$ is nonpositive, the continuity path is
\begin{align}\label{kegc}
 \log \omega_{\vphi}^{n} = \log \omega^{n} -\l \vphi + tf. \;
\end{align}
While when the $\l$ is positive, i.e. the Fano case, the Aubin path is 
\begin{align}\label{kegcc}
 \log \omega_{\vphi}^{n} = \log \omega^{n} -\l t \vphi + f. \;
\end{align}
When we solve this equation by the continuity method, we need to derive the a priori $C^{2,\a}_\b$ estimate as following. 
\begin{prop}\label{ke schauder}
Assume that the solutions of \eqref{kegc} and \eqref{kegcc} have up to the second order estimate and $0<\b_1<\frac{2}{3}$.
Then the following estimate holds on any small ball $B\subset \mathfrak X$
\begin{align}
\vphi\in C^{2,\a}_\b(B) \;.
\end{align}
\end{prop}
\begin{proof}
Applying the proposition above with dimension $n$, it suffices to check $|\p\log\tilde\om^{n}|_{L^q}$,
$|\log\om^{n}|_{C^\a_\b}$,
$|\p \tilde F|_{L^q}$, $|F|_{C^\a_\b}$. Since $\om\in C^{2,\a}_\b$, so we have $\log\om^{n}\in C^\a_\b$. The proof of
\lemref{lemma: connection bounded} implies that when $(\frac{2}{\b_1}-3)q+2\b_1>0$,
$|\p\log\om^{n}|_{L^q}$ is bounded. Thus the result follows from the next lemma.
\end{proof}
\begin{lem}
$|\p f|_\om$ is bounded when $0<\b_1<\frac{2}{3}$.
\end{lem}
\begin{proof}
In local coordinate, 
we have $f=-\log[\om^n|s|^{2(1-\b_1)}]+\log\om^n_0-\l\delta|s|^{2\b_1}+f_0.$
Note that $\om_0$ and $f_0$ are both smooth. Moreover, $|s|^{2\b_1}=\rho |z^1|^{2\b_1}=\rho \circ W^{-1} \cdot |w^1|^{2}$, so its first derivative is bounded. It remains to verify that the first derivative of $\log(\om^n|s|^{2(1-\b_1)})$ with respect to $\om$ is bounded. Since $\om^n|s|^{2(1-\b_1)}$ is positive and bounded, it suffices to prove that $|\p(\om^n|s|^{2(1-\b_1)})|_\om$ is bounded. Put the weight into the matrix $\om$, we have a new metric $\om_1$, \begin{align*}
\frac{2}{i}\om_1&=|s|^{2(1-\b_1)}(g_{0})_{ 1\bar 1}dz^1\wedge dz^{\bar 1}
+(g_{0})_{ k\bar l}dz^k\wedge dz^{\bar l}\\
&+|s|^{(1-\b_1)}(g_{0})_{ k\bar 1}dz^k\wedge dz^{\bar 1}
+|s|^{(1-\b_1)}(g_{0})_{ 1\bar l}dz^1\wedge dz^{\bar l}\\
&+|z^1|^{2\b_1}\rho_{k\bar l} dz^k\wedge dz^{\bar l}\\
&+ |s|^{(1-\b_1)}\b_1|z^1|^{2(\b_1-1)} (z^1\rho_{k} dz^k\wedge dz^{\bar 1}
+z^{\bar 1} \rho_{\bar l} dz^1 \wedge dz^{\bar l})\\
&+ |s|^{2(1-\b_1)}\b_1^2 \rho|z^1|^{2(\b_1-1)} dz^1\wedge dz^{\bar 1}
\end{align*}for $k,l$ from $2$ to $n$.
The components of $\om_1$ under the variables $w^i$ become
\begin{align*}
\tilde g_{1\bar 1}
&=\frac{1+\b_1^2}{2\b_1^2}\rho^{\frac{1-\b_1}{\b_1}}| w^1|^{\frac{2}{\b_1}-2}[g_{01\bar 1}\circ W^{-1} | w^1|^{\frac{2}{\b_1}-2}
+|w^1|^{\frac{2}{\b_1}}\rho_{1\bar 1}
+\b_1(w^1\rho_1+w^{\bar 1}\rho_{\bar 1}+\b_1^2\rho)],\\
\tilde g_{1\bar l}
&=\frac{1+\b_1}{2\b_1}\rho^{\frac{1-\b_1}{2\b_1}}| w^1|^{\frac{1}{\b_1}-1}
[|w^1|^{\frac{1}{\b_1}-1}g_{01\bar l}\circ W^{-1}
+|w^1|^{\frac{1}{\b_1}+1}\rho_{1\bar l}\circ W^{-1}
+\b_1w^1\rho_{\bar l}\circ W^{-1}],\\
\tilde g_{k\bar l}&=g_{0k\bar l}\circ W^{-1}
+|w^1|^2\rho_{k\bar l}\circ W^{-1}.
\end{align*}
Now, we check one by one the first derivative with respect to $w^i$. The first derivative of $\tilde g_{k\bar l}$ follows from \lemref{lemma: connection bounded}. Note that $\rho$ is smooth on $w^k$ for $1\leq k\leq n$.
\begin{align*}
&\frac{\p}{\p w^1}\tilde g_{1\bar 1}
=O(| w^1|^{\frac{4}{\b_1}-5}+| w^1|^{\frac{2}{\b_1}-3});\\
&\frac{\p}{\p w^i}\tilde g_{1\bar 1}=\frac{\p}{\p w^i}\tilde g_{1\bar l}
=O(1).
\end{align*}
Now let us check
$\frac{\p}{\p w^1}\tilde g_{1\bar l}.$
It contains three terms. The first term is
\begin{align*}
&\frac{\p}{\p w^1}(|w^1|^{\frac{2}{\b_1}-2}g_{01\bar l}\circ W^{-1})\\
&=\frac{\p}{\p w^1}[(|w^1|^{\frac{2}{\b_1}-1}|w^1|)(|w^1|^{-1}g_{01\bar l}\circ W^{-1})].
\end{align*}
Since $g_{01\bar l}\circ W^{-1}$ is also smooth and converges to zero as $w^1$ goes to zero,
so the growth rate of this term is $O(|w^1|^{\frac{1}{\b_1}-1})$.
The second and third term are both $O(1)$.
Thus this lemmas follows.
\end{proof}
\subsection{Boundary Schauder estimate: $\tau>0$}
We adapt Krylov's method \cite{MR688919} (also c.f. \cite{MR1814364})
for the boundary estimate to our cone case.
We notice that the linear equation is of divergence form,
so the Harnack inequality and maximum principle proved in the next section can be applied here. The boundary of $\mathfrak X$ is $X\times \p R$, which is a manifold with $2n+1$ real dimension. Under the local coordinate ${z^i=x^{n+1}+i y^{n+1}}$, the boundary is defined by $x^{n+1}=0$. Denote $x'=\{x^1, y^1, \cdots, x^n, y^n, y^{n+1}\}.$
\begin{prop}
Assume $0<\b_1<\frac{2}{3}$ and that $\vphi_0,\vphi_1$ are in $\mathfrak{I_i}$, $i=1,2,3$.
Then the $C^{1,1}_\b$ solution $\Psi$ of the approximate
geodesic equation \eqref{per equ} belongs to $C^{2,\a}_\b$ on the boundary of $\mathfrak X$.
\end{prop}
\begin{proof}
Recall the approximate geodesic equation is
\begin{equation}\label{per equ sim bdy}
  \left\{
   \begin{array}{ll}
\log\det(\Om_{\Psi_1 i\bar j}+\tilde\Psi_{i\bar j})=h
=\log\tau+ \tilde\Psi + \log\det(\Om_{\Psi_1i\bar j}) &\text{ in }\mathfrak{M}\; ,\\
\tilde\Psi(z)=0 & \text{ on }\p\mathfrak{X}\; .
   \end{array}
  \right.
\end{equation}
We first see that the tangent-tangent direction of the boundary estimate equals to the same estimate of the boundary values.
Then the normal-normal estimate follows from the approximate geodesic equation
$$
[\vphi''-(\p\vphi',\p\vphi')_{g_{\vphi}}]
\det\om_\vphi=\Om_\Psi^{n+1}=\tau e^{\Psi-\Psi_1} \det(\Om_{1i\bar j})
$$ with the estimates of the tangent-normal direction and the tangent-tangent direction.
We differentiate \eqref{eq_starting_@_interior_holder_estimate} with respect to $\p_k$ for a fixed $k\in 1,\cdots, n$,
and we get
\begin{align*}
\tri' \tilde\Psi_k
 = h_k - g_\Psi^{i\bar j}g(\Psi_1)_{ i\bar j k}\;.
\end{align*}
We use the flat metric as the weighted metric to derive the differential equation of $u=\sqrt{g^{k\bar k}} \tilde\Psi_k$.
Then we obtain that $u$ satisfies
\begin{align*}
\tri' u
 =\sqrt{g^{k\bar k}}( h_k - g_\Psi^{i\bar j}g(\Psi_1)_{ i\bar j k})\;.
\end{align*}
We denote the right hand side as $f$. According to \lemref{connection}, $f$ is bounded when $0<\b_1<\frac{2}{3}$
and $\vphi_0,\vphi_1\in C^3_\b$. Note that $u$ vanishes on the boundary $\p \mathfrak  X$.
We fix a point $p$ on the boundary, and we take coordinates ${z^i}$ centered at $p$.
We introduce the following domains for a small radius $d$.
\begin{align*}
|x'|^2_{\b_1}&=|z^1|^{2\b_1}+\sum_{i=2}^n |z^i|^2+|y^{n+1}|^2,\;
B_d(p)=\{z\in M\vert |z-p|_{\b_1}\leq d\},\\
B_{1}&=B_d(p)\times \{x^{n+1}\vert 0\leq|x^{n+1}|\leq\delta d, x^{n+1}\geq 0\},\\
B_{1}^3&=B_d(p)\times \{x^{n+1}\vert \delta d\leq|x^{n+1}|\leq3\delta d, x^{n+1}\geq 0\},\\
S_2&=B_d(p)\times\{|x^{n+1}|=2\delta d, x^{n+1}\geq 0\},\\
B_{2}&=B_{2d}(p)\times\{x^{n+1}\vert 0\leq|x^{n+1}|\leq2\delta d, x^{n+1}\geq 0\},\\
B_{4}&=B_{4d}(p)\times \{x^{n+1}\vert 0\leq|x^{n+1}|\leq4\delta d, x^{n+1}\geq 0\}.
\end{align*}
Here, $\delta\ll 1$ is a small positive constant such that $v:=\frac{u}{x^{n+1}}$ is strictly positive on $S_2$.
We  assume that $v$ is nonnegative on $B_4$; then $u\geq 0$.

We use the barrier function
$$w=[(4-\frac{|x'|^2_{\b_1}}{d^2})\inf_{S_2}v
+(1+d\sup|f|)\frac{\frac{x^{n+1}}{2d}-\delta}{\sqrt{\delta}}
]x^{n+1}.$$
We first prove that on the boundary of $B_2$, $w\leq u$.
On $|x^{n+1}|=2\delta d$, we have
$w\leq 4 x^{n+1}\inf_{S_2}v \leq
 u$;
on $|x^{n+1}|=0$, we have $w=0\leq u$;
on $|x'|^2_{\b_1}=2d$, $w\leq 0\leq u$.
Then, in $B_2$ we compute $\tri' w=-\frac{\inf_{S_2}v}{d^2}x^{n+1}
+(1+d\sup|f|)\frac{1}{2d\sqrt{\delta}}\geq f $. According to the maximum principle \lemref{max prin},
we have $w\leq u$ on $B_2$. As a result, we obtain in $B_1$,
\begin{align}\label{section set}
v&\geq (4-\frac{|x'|^2_{\b_1}}{d^2})\inf_{S_2}v
+(1+d\sup|f|)\frac{\frac{x^{n+1}}{2d}-\delta}{\sqrt{\delta}}\nonumber \\
&\geq 2 \inf_{S_2}v -d\sup|f|.
\end{align}
Note that $\delta$ only needs to be an arbitrarily small constant.

Now, notice that $\tri' u$ is of the divergence form, we could apply the interior Harnack inequality (Proposition \ref{linearharnack}) to $\tri' u= f$ on $B_1^3$;
since now $\frac{u}{3\delta d}\leq v\leq \frac{u}{\delta d}$ we obtain
\begin{align*}
\sup_{B_1^3} v
&\leq C(\inf_{B_1^3} v+\sup | f|).
\end{align*}
Here $C$ depends on $\om$.
Since $\inf_{B_1^3} v\leq \inf_{S_2}v$, using \eqref{section set}, we have
\begin{align}\label{resulting inequ}
\sup_{B_1^3} v\leq  C(\inf_{B_1}v+d\sup |f|).
\end{align}
Replacing in the former arguments, $v$ by $v-\inf_{B_4} v$ and then by $\sup_{B_4} v - v$, noticing that they are
both positive, and finally adding the resulting inequalities \eqref{resulting inequ}, we arrive at the following inequality,
\begin{align*}
\osc_{B_1} v\leq  \frac{C-1}{C}\osc_{B_4}v+ 2d\sup |f|.
\end{align*}
Then by the iteration Lemma 8.23 in \cite{MR1814364}, we have the H\"older estimate of $v$ for any $d\leq d_0$,
\begin{align*}
\osc_{B_d} v\leq  C\frac{d^\a}{d_0^\a}(\osc_{B_{ d_0}} v+d_0\sup |f|).
\end{align*}
For any $q$ in $\mathfrak X$, choose $d=|p-q|_{\b_1}$ and $d_0=diam(\mathfrak X)$, we obtain the H\"older continuity of $v$ as
\begin{align*}
\frac{\vert v(p)-v(q) \vert}{|p-q|^\a_{\b_1}}\leq  C(d_0^{-\a}\sup_{B_{ d_0}} |v|+d_0\sup |f|).
\end{align*}
Since $u$ vanishes on the boundary and depends trivially on the variable $y^{n+1}$, we have $\p_{z^{n+1}} u$ is $C^\a_\b$.
Thus the proposition is proved.
\end{proof}

\subsection{Uniqueness of the $C^{1,1}_\b$ cone geodesic}\label{unigeo}
In Theorem \ref{sole geo}, we have obtained the existence
of a $C^{1,1}_\b$ cone geodesic.
Our present goal is to prove its uniqueness.
Suppose that $\Phi_i$ for $i=1,2$ are two cone geodesic segments,
which correspond to the solutions $\Psi_{\tau_i}\in C^{2,\a}_\b$ of
\begin{equation*}
  \left\{
   \begin{array}{ll}
\frac{\det(\Om_{\Psi_{\tau_i}})}{\det(\Om_1)}= \tau_i e^{a(\Psi_{\tau_i} - \Psi_1)}&
\text{ in }\mathfrak{M}\; ,\\
\Psi_{\tau_i}=\Psi_{0i} & \text{ on }\p\mathfrak{X}\; ,
   \end{array}
 \right.
\end{equation*}
for $i=1,2$ and $\tau_i \in [0,1]$.
Since $\Psi_{\tau_i} \to \Psi_i$ in $C^{1,\a}_\b$ as $\tau_i \to 0$, then for any $\eps >0$
we can find two values $\tau_1 ,\,\tau_2$ such that
\begin{align*}
\sup_{\mathfrak{X}}|\Psi_i-\Psi_{\tau_i}| \leq \eps \;.
\end{align*}
So, we compute
\begin{align*}
\log\det(\Om_{\Psi_{\tau_1}}) - \log\det(\Om_{\Psi_{\tau_2}}) =
\int_0^1 g_t^{i\bar j} dt (\Psi_{\tau_1} - \Psi_{\tau_2})_{i\bar j}
>a(\Psi_{\tau_1} - \Psi_{\tau_2})\;,
\end{align*}
where $g_t = t g_{\Psi_{\tau_1}} + (1-t) g_{\Psi_{\tau_2}}$ and $a\geq 0$.
Now, applying \lemref{wmp} we have,
\begin{align*}
\sup_{\mathfrak{X}}(\Psi_{\tau_1}-\Psi_{\tau_2})
\leq \sup_{\p \mathfrak{X}}(\Psi_{01}-\Psi_{02})\;.
\end{align*}
So we have
\begin{align*}
\sup_{\mathfrak{X}}(\Psi_{1}-\Psi_{2})
&\leq \sup_{ \mathfrak{X}}(\Psi_{\tau_1}-\Psi_{1})+
 \sup_{ \mathfrak{X}}(-\Psi_{\tau_2}+\Psi_{\tau_1})+
 \sup_{ \mathfrak{X}}(\Psi_{\tau_2}-\Psi_{2})\\
&\leq 2\eps +  \sup_{\p \mathfrak{X}}(\Psi_{01}-\Psi_{02})\;.
\end{align*}
Then, switching $\Psi_1$ and $\Psi_2$ and letting $\eps\to 0$, we end up with
\begin{align*}
 \sup_{\mathfrak{X}} |\Psi_{1}-\Psi_{2}|
\leq \sup_{\p \mathfrak{X}}|\Psi_{01}-\Psi_{02}| \;.
\end{align*}
The above inequality proves the uniqueness
of a cone geodesic segment with prescribed boundary values.


\section{Linearized equation}\label{linear thy}
In this section we consider the general linear elliptic equation
\begin{equation}\label{linear equ ge}
  \left\{
   \begin{array}{rl}
Lv&=g^{i\bar j}v_{i\bar j}+b^iv_i+cv=f+\p_ih^i\\
v&=v_0
   \end{array}
  \right.
\end{equation}
in the space $(\mathfrak{X},\mathfrak{D})$ defined in Section \ref{space}.
Here $g^{i\bar j}$ is the inverse matrix of a K\"ahler cone metric $\Om$ in $H^{2,\a}_\b$.
Moreover, we are given the following datas.
\begin{align}\label{linear equ ge co}
b^i, h^i\in C^{1,\a}_\b; c, f\in C^\a_\b \text{ and } v_0\in C^{2,\a}_\b \;.
\end{align}
This type of equation has been studied via the general edge calculus theory (c.f. Mazzeo \cite{MR1133743} and references therein). However, We consider in this paper the K\"ahler manifold with boundary. The edge space is not defined near the boundary.
Recently, Donaldson introduced a function space on a closed K\"ahler manifold which fits well with our geometric problem. In Section \ref{space}, Definition \ref{defn: 3dribdy}, we generalized DonaldsonÕs space to the boundary
case and thus introduced a H\"older space. Now we studied \eqref{linear equ ge} this H\"older space. We collect here the analytic results on the linear equation \eqref{linear equ ge} which are not only used in previous arguments above but also for our further applications.

\subsection{The maximum principle and the weak solution}\label{maximum_principle}
We say $v$ is the solution of \eqref{linear equ ge} if it satisfies this equation on $\mathfrak X\setminus \mathfrak D$ and belongs to $C^{2,\a}_\b$. From the theory of the elliptic equation, we know that $V$ is smooth outside $\mathfrak D$. The delicate part here is always the estimate near the divisor.
We first prove a maximum principle for the Kahler cone metric.
\begin{lem}\label{max prin}
Assume that $v$ satisfies $Lv\geq 0$ (resp. $Lv\leq 0$) with $c<0$, then the maximum (minimum) is achieved on the boundary
i.e.
$$
\sup_\mathfrak{X} v=\sup_{\p\mathfrak{X}\setminus\p\mathfrak{D}}v
\qquad \left(\inf_\mathfrak{X} v=\inf_{\p\mathfrak{X}\setminus\p\mathfrak{D}}v\right)\;.
$$
\end{lem}
\begin{proof}
Set $u=v+\eps S$ and $S=\| s \|^{2\kappa}$ with $(1+\a)\b>2\kappa\geq\b$. Then $|\p S|_g$ is bounded. Suppose that  $p$ is the maximum point of $u$. According to \lemref{auxiliary function}, $p$ cannot be on $\mathfrak{D}$.
So either $p$ stays on the boundary $\p\mathfrak{X}\setminus\p\mathfrak{D}$ or in the interior
of $\mathfrak{X}\setminus\mathfrak{D}$.
Then in the latter case, at the maximum point $p$ we have
$$
0\leq Lv=Lu-\eps LS\leq  cu-\eps(\tri_g S+b^iS_i+cS)\leq cu + \eps C\;.
$$
Here we use $b^iS_i\geq -|b^i|^2_g-|\p S|^2_g$ and the first conclusion in \lemref{auxiliary function},
$\tri_g S\geq -C\;.$
Combining these inequalities we obtain
$$
u(p)\leq \eps C \;.
$$

Then at any point $x\in \mathfrak{X}$, we have the following relation
\begin{align*}
v(x)= u(x)-\eps S\leq u(p) \leq \sup_{\p\mathfrak{X}\setminus\p\mathfrak{D}}v
+\eps C \;,
\end{align*}
since $S$ is nonnegative.
Similarly, Similarly, we shall be using $u=v-\eps F$ instead for $Lv \leq 0$.
As a result, the proposition follows as $\eps\rightarrow0$.
\end{proof}

Now we use the maximum principle to deduce the uniqueness of solutions of
the elliptic equation \eqref{linear equ ge}.
\begin{cor}
If $v_1$, $v_2$ are two solutions of the linearized equation
\eqref{linear equ ge} with $c<0$, then $v_1=v_2$.
\end{cor}

The singular volume form
$\om^n$ with respect to the cone metric gives a measure on the manifold $\mathfrak{X}$.
As a consequence, the $L^p(\mathfrak{X},g)$ space is defined in the usual way.
The $W^{1,p}(\mathfrak{X},g)$ space furthermore requires that
the derivatives satisfy $\int_\mathfrak{X}|\nabla f|^p_\Om\om^n<\infty$.

\begin{defn}
The weak solution in $W^{1,2}$ of \eqref{linear equ ge} is defined, for any $\eta\in W^{1,2}_0$, in the sense of distributions;
\begin{align}\label{weak sol}
\mathcal{L}(v,\eta)=\int_{\mathfrak{X}}[g^{i\bar j}v_i\eta_{\bar j}-b^iv_i\eta-cv\eta]\om^n
=\int_{\mathfrak{X}}-\eta f- h^i\eta_i\om^n.
\end{align}
Note that our weak solution is defined globally.
\end{defn} The following lemmas follow directly from the local lifting $P\circ W$ (cf. \eqref{euc}).
\begin{lem}(Sobolev imbedding)\label{sob}
Assume that $f\in W^{1,2}_0$. Then there is a constant $C$ depending on $n,\b$ such that
\begin{align*}
||f||_{\frac{2n}{n-1}}\leq C||f||_{W^{1,2}}\;.
\end{align*}
\end{lem}
\begin{lem}\label{kon}(Kondrakov compact imbedding)
The imbedding $W^{1,2}_0\rightarrow L^p$ for $1\leq p<\frac{2n}{n-1}$ is compact.
\end{lem}

\begin{lem}\label{wmp}(Weak maximum principle)
Let $v\in W^{1,2}$ satisfy $Lv\geq 0(\leq0)$ in $\mathfrak{X}$ with $c\leq 0$. Then
$$
\sup_\mathfrak{X}v\leq\sup_{\p\mathfrak{X}}v^+
\qquad \left(\inf_\mathfrak{X}v\geq\sup_{\p\mathfrak{X}}v^-\right)\;.
$$
\end{lem}
\begin{proof}
From the definition of weak solution we have that $Lv\geq 0$ implies $\mathcal{L}(v,\eta)\leq 0$.
Then for $\eta\geq0$, we have
\begin{align*}
\int_{\mathfrak{X^+}}[g^{i\bar j}v_i\eta_{\bar j}-b^iv_i\eta]\om^n\leq 0\;,
\end{align*}
where  $\mathfrak{X}^+=\{x\in\mathfrak{X}\vert v(x)\geq0\}$. Let $v^+=\max\{0,v\}$.
If $b^i=0$, letting $\eta=\sup\{0,v-\sup_{\p\mathfrak{X}}v^+\}$, we have
$$
\int_{\mathfrak{X^+}}|\nabla\eta|^2\om^n\leq0 \;.
$$
So $|\nabla\eta|^2=0$ on $\mathfrak{X^+}\setminus\mathfrak{D}$.
Since $\eta=0$ at the maximum point on the boundary of $\mathfrak{X^+}$,
we obtain $\eta=0$ on $\mathfrak{X^+}\setminus\mathfrak{D}$.
Since the measure of $\mathfrak{D}$ is zero, we could modify the value of $\eta$ such that $\eta=0$
on the whole $\mathfrak{X}$. Then the lemma follows for $b^i=0$.
When $b^i\neq0$, using the Sobolev inequality \eqref{sob},
the proof is the same as that of Theorem 8.1 in \cite{MR1814364}.
\end{proof}
Then this lemma and a standard argument by means of the Fredholm alternative theorem implies the uniqueness and the existence of the weak solution.
\begin{prop}\label{ext weak}
The linear equation \eqref{linear equ ge} with $c\leq 0$ has a unique weak solution in $W^{1,2}$.
\end{prop}

\subsection{H\"older estimates}
We remark that in this subsection, all results hold for normal-crossing divisors $D$ with more than one component. However, we just check for one component. The general multiple case follows from the normal crossing condition. The H\"older estimates derived in this subsection are used iin the proof of  both the interior and boundary Schauder estimates of the approximate geodesic equation.
Before stating the proposition on the global and local boundedness,
we require some technical lemmas which will be useful later.
Denote $\omega_0 = dz^1 \wedge d\overline{z}^1 + \cdots +  dz^{n+1} \wedge d\overline{z}^{n+1} $.
Then locally in a neighborhood $U_p$ near $p\in \mathfrak D$, $\om_0^{n+1} = n! \cdot dz^1 \wedge d\overline{z}^1
\wedge \cdots \wedge dz^{n+1}
\wedge d\overline{z}^{n+1} , $ and then we have that there is a bounded function $h$ such that
$$
\om^{n+1} = \beta^2 |z^1|^{2(\beta - 1)} \om_0^{n+1} e^h \; .
$$
Finally, let $m=2n+2$.
\begin{lem}\label{lemma three}
There is a constant $C$ depending on $|h|_\infty$ such that, for any
$s > \frac{1}{\beta}$, the following inequality holds
\begin{align*}
\left( \int_{U_p} f^p \om^{n+1} \right)^{\frac{1}{p}}
\leq C \left( \int_{U_p} f^{sp} \om_0^{n+1} \right)^{\frac{1}{sp}}\;.
\end{align*}
\end{lem}
\begin{proof}
Let $z^1 = \rho e^{i \theta}$ and compute
\begin{align*}
&\left( \int_{U_p} f^p \om^{n+1} \right)^{\frac{1}{p}}
= \left( \int_{U_p(z')} \int_0^{r_2} \int_0^{2\pi} f^p \beta^2 \rho^{2(\beta -1)}
e^h \om_0^{n+1} \right)^{\frac{1}{p}}
\\ &\leq  \left( \int_{U_p(z')} \int_0^{r_2} \int_0^{2\pi} f^{sp}
e^h \om_0^{n+1} \right)^{\frac{1}{sp}}
\cdot
 \left( \int_{U_p(z')} \int_0^{r_2} \int_0^{2\pi} (\rho^{2\beta - 2})^t
e^h \om_0^{n+1} \right)^{\frac{1}{tp}}\;.
\end{align*}
Here $\frac{1}{s}
+ \frac{1}{t} =1$. Since $t < \frac{1}{1-\beta}$, the second term is bounded, we have
$s> \frac{1}{\beta}$, which concludes the proof.
\end{proof}

\begin{lem}\label{lemma four}
There is a constant $C$ depending  on  $ \beta $ and $|h|_\infty$ such that, for any $s > 1$,
the following formula holds
\begin{align*}
\left( \int_{U_l} f^p \om_0^{n+1} \right)^{\frac{1}{p}} \leq C \left( \int_{U_l} f^{sp}
\om^{n+1} \right)^{\frac{1}{sp}} \;.
\end{align*}
\end{lem}
\begin{proof}
Again we compute in polar coordinates
\begin{align*}
&\left( \int_{U_p} f^p \om_0^{n+1} \right)^{\frac{1}{p}}
= \left( \int_{U_p(z')} \int_0^{r_2} \int_0^{2\pi} f^p  \rho^{\frac{2(\beta -1)}{s}}
 \rho^{-\frac{2(\beta -1)}{s}} \om_0^{n+1} \right)^{\frac{1}{p}}
\\ &\leq  \left( \int_{U_p} f^{sp} \beta^{-2} e^{-h}\om^{n+1} \right)^{\frac{1}{sp}}
 \left( \int_{U_p(z')} \int_0^{r_2} \int_0^{2\pi}
\rho^{-\frac{2(\beta - 1)t}{s}}  \om_0^{n+1}  \right)^{\frac{1}{tp}}\;.
\end{align*} Here $s,t$ are two positive constants such that $\frac{1}{s}
+ \frac{1}{t} =1$.
The second term is bounded as $\frac{t}{s} > \frac{1}{\beta - 1}$ which is trivially satisfied.
\end{proof}

The proof of the following propositions
are in the same vein as the proofs in Chapter 8 in \cite{MR1814364}.
However, by the lemmas stated above, we need
a careful analysis in the charts which intersect the divisor.
\begin{prop}\label{glob bound claim}(Global boundedness)
If $v$ is a $W^{1,2} $ sub-solution (respectively super-solution) of
\eqref{linear equ ge} in $\mathfrak X$ satisfying $v\leq 0$ (resp.$ v \geq 0$) on $\partial \mathfrak X$;
moreover, if $f\in L^{\frac{q}{2}}$ and $h^i\in L^{q}$, $i=1,\cdots, n+1$ with $q> m$
then there is a constant $C$ depending on $|b^i|_g$, $|c|_\infty$, $q$, $\b$ such that
\begin{align*}
&\sup_{\mathfrak X} v(-v) \leq C(\| v^+(v^-) \|_2 + \| f \|_{\frac{q}{2}}+\|h^i\|_q).
\end{align*}
\end{prop}
\begin{proof}
Assume that $v$ is a $W^{1,2}$ sub-solution of \eqref{linear equ ge}.
We are going to use the De Giorgi-Nash-Moser iteration as in Theorem $8.15$ in \cite{MR1814364}.
Denote $k=\| f \|_{\frac{q}{2}}+\|h^i\|_q$. Choose $w=v^++ k$ and $\eta=\int_{k}^wa^2s^{2(aa-1)}ds$ for $a\geq 1$ in $\mathcal{L}(v,\eta)$. With the Sobolev inequality Lemma \ref{sob}, we have
$$
\| w \|_{\frac{2na}{n-1}; \om} \leq  (C(a+1))^{\frac{1}{a}} \| w \|_{2a; \om}\; .
$$
We use Lemma \ref{lemma three} and Lemma \ref{lemma four} on the coordinates which intersect
the divisor $\mathfrak D$ and the H\"older inequality in the remainder coordinates.
After patching them together via a partition of the unity we have,
for $ s > \frac{1}{\beta} \geq 1$,
$$
\| w \|_{\frac{2nas}{n-1}; \om_0} \leq  (C(a+1))^{\frac{1}{a}} \| w \|_{2as ; \om_0}\;.
$$
Now we follow a standard iteration argument;
using the interpolation inequality we have with $\chi=\frac{n}{n-1}$
$$
\| w \|_{\frac{\chi^{N 2s}}{n-1}; \om_0} \leq C \| w\|_{\frac{2}{s}; \om_0}\;.
$$
Finally, letting $N\rightarrow \infty$ and using Lemma \ref{lemma four} again, we get the proposition.
\end{proof}

Denote as $d$ the distance measured via the K\"ahler cone metric $\om$
\begin{prop}\label{loc est}(Local boundedness)
Suppose that  $v$ is a $W^{1,2}$ sub-solution of \eqref{linear equ ge} and suppose that
$f \in L^{q}, $ and $h^i\in L^{q}$, $i=1,\cdots, n+1$ with $ q > m$.
Then for any ball $B_{2d}(y) \subset  \mathfrak X$ and any $p>1$
there is a constant $C$ depending on $(|b^i|_g+|c|_\infty)d$, $q$, $\b$, $p$ such that
\begin{equation*}
\sup_{B_d(y)} v(-v) \leq C ( d^{-\frac{m}{p}} \| v^+(v^-) \|_{L^p(B_{2d}(y))}
+ d^{2(1- \frac{m}{2q})} \| f\|_{\frac{q}{2}}+d^{1-\frac{m}{q}}\|h^i\|_q )\;.
\end{equation*}
\end{prop}
\begin{proof}
We will prove the local boundedness of the homogeneous equation. The general case follows by means of using $v+d^{2(1- \frac{m}{2q})} \| f\|_{\frac{q}{2}}+d^{1-\frac{m}{q}}\|h^i\|_q$ instead of $v$. Then $v$ would be a weak sub-solution of \eqref{linear equ ge} with $f=0$ and $h^i=0$; namely
$\mathcal{L}(v,\eta)\leq 0 .$ Assume $d=1$
and take the test function to be $\eta^2 v^\a$ for $\eta \in C_0^1 (B_4)$ and $\a >0.$
Then we have for $w:= v^{\frac{\a+1}{2}}$
\begin{align*}
\| \eta w \|_{\frac{2n}{n-1};\om} \leq C \cdot (\|w\p \eta \|_{2;\om} + \| w\eta \|_{2;\om}  )\; .
\end{align*}
Using Lemma \ref{lemma three} and Lemma \ref{lemma four} we obtain,
on any open set $U_p$ which intersects the divisor $D$ for $s > \frac{1}{\beta}$
\begin{align*}
\| \eta w \|_{\frac{2n}{s(n-1)};\om_0}
\leq C [  \|w \p\eta \|_{2;\om}
+\| w\eta \|_{2s;\om_0}]\; .
\end{align*}
We claim for the first addendum on the right hand side it holds $ \|w\p\eta \|_{2;\om} \leq  C \|w\p\eta \|_{2s;\om_0}$ with $s> \frac{1}{\beta}$.
Again, by means of Lemma \ref{lemma three} and Lemma \ref{lemma four} we compute
\begin{align*}
& \|w\p\eta \|_{2;\om}
= \left[ \int_{U_p} w^2 (\partial_{z_1} \eta \partial_{z_{\overline{1}}} \eta |z^1 |^{2(1-\beta)}
\frac{1}{\beta^2}
+
\sum_{i=2}^{n+1} \partial_{z_i} \eta \partial_{z_{\overline{i}}} \eta) \om^{n+1} \right]^\frac{1}{2}
\\ &\leq \left[ \int_{U_p} w^2 (\partial_{z_1} \eta \partial_{z_{\overline{1}}} \eta ) e^h \om_0^{n+1}
+ C \left( \int_{U_p} \sum_{i=2}^{n+1} w^{2s} (\partial_{z_i} \eta \partial_{z_{\overline{i}}} \eta)^s
\om^{n+1} \right)^\frac{1}{s} \right]^\frac{1}{2}
\\ &\leq C \left(  \int_{U_p} w^{2s} |\p\eta|_{\om_0}^{2s}  \om_0^{n+1}  \right)^{\frac{1}{2s}}\; ,
\end{align*}
where to get  the last step we used the H\"older inequality on the first term.
So standard argument with Lemma \ref{lemma four} implies
\begin{align*}
 \| v\|_{\infty ; B_1 , \om_0 }
\leq C \| v \|_{ ps ; B_2 , \om_0 } \leq   C \| v \|_{ ps^2 ; B_2 , \om } \; .
\end{align*}
The local boundedness follows from the next observation;
$B_1 (0 , \om) \subset B_1 (0 , \om_0) $ which follows from
the distance inequality,
$$
\sqrt{ |z^1 |^2 + \sum_{i=2}^n |z^i|^2  }
 \leq \sqrt{ |z^1 |^{2\beta} + \sum_{i=2}^n |z^i|^2  }  \leq 1 .
$$
\end{proof}


\begin{prop}(Weak Harnack inequality)\label{harnack}
Suppose that $v$ is a  $W^{1,2}$ super-solution of \eqref{linear equ ge},
non-negative in a ball $B_{4d}(y) \subset \mathfrak X$ and suppose that  $f\in L^{\frac{q}{2}}$ and $h^i\in L^{q}$, $i=1,\cdots, n+1$ with $q >m$.
Then, for any $\frac{n+1}{n}>p>1$ there is a constant $C$ depending on $(|b^i|_g+|c|_\infty)d$, $q$, $\b$, $p$ such that
\begin{align}\label{eqharnack}
d^{-\frac{m}{p}} \| v \|_{L^p(B_{2d}(y))}
\leq
C
\left\{
\inf_{B_{d(y)}}v+d^{2(1-\frac{m}{q})}||f||_{\frac{q}{2}} +d^{1-\frac{m}{q}}\|h^i\|_q
\right\}\;.
\end{align}
\end{prop}
\begin{proof}
We assume $d=a$ and argue as in the proof of the local boundedness with different test function.
Thus it suffices to prove, for the weak super-solution of \eqref{linear equ ge} with vanishing right hand side, there is a $p > 0$ and constant $C$ such that
\begin{align}\label{reduced claim harnack}
  \int_{B_2} v^{-p} \om^{n+1} \int_{B_2} v^{p} \om^{n+1} \leq C \;.
\end{align}
Choose a test function of the form $\eta^2 v^{\a}$ and let $w:= \log v$ and $\a =-1$. Here $\eta$ is the cut-off function defined in \lemref{lemma: integration by parts}.
We have by the CauchyÐSchwarz's inequality for small $\eps_1$ and $\eps_2$,
$$
\int_{B_r} | \p w |^2 \om^{n+1} \leq   \frac{2}{\epsilon_1} \int_{\mathfrak X} |\p \eta|^2 \om^{n+1}
+2\left( \frac{|b^i|_0}{4\epsilon_2} + |c|_0 \right) \int_{\mathfrak X} \eta^2 \om^{n+1} \; .
$$
Since $(\mathfrak X ,\om)$ has finite volume, the second term is bounded.
Concerning the first addendum, we compute,
\begin{align*}
\int_{U_p} |\p\eta|^2 \om^{n+1}
\leq C \int_0^{2\pi}\int_0^r t^{4-2\b+2(\b-1)} dt d\theta
\leq Cr^{3}.
\end{align*}
We conclude that $\int_{B_r} | \p w |\om^{n+1}$ is bounded.

Next we claim that
$\int_{B_r} | \p w |_0\om_0^{n+1}$ is also bounded.
To prove the claim, let's compute
\begin{align*}
\int_{B_r} | \p w |_0 \om_0^{n+1}  =  \int_{B_r}(|\partial_{z^1 } w |_0^2 +\sum_{i=2}^{n+1}|\partial_{z^i } w |_0^2 )^{\frac{1}{2}} \beta^{-2} | z^1 |^{2(1-\beta)} e^{-h}\om^{n+1}.
\end{align*}
The second is bounded, since $h$ and $| z^1 |$ are bounded.
For the first addendum,
we first consider the case when  $ |   \partial_{z^1 } w  |_0  \leq 1$. The its boundedness follows from the finiteness of the volume.
The second case is when   $ |   \partial_{z^1 } w  |_0 > 1$. In this second case $
 |   \partial_{z^1 } w  |_0 <  |
\partial_{z^1 } w  |_0^2
$
and so its integral is bounded by
$\int_{B_r} | \p w |^2 \om^n .$
The claim thus holds.
Now we apply the Moser-Trudinger inequality (see Theorem $7.21$ in \cite{MR1814364}) with respect to $\om_0$.
Thus there exists a constant $p_0$
 such that
$$
\int_{B_3} e^{p_0 | w - w_0   |  } \om_0^n
$$
is bounded and so is
$$
\int_{B_3} v^{p_0} \om_0^n \int_{B_3}v^{-p_0   } \om_0^n \;.
$$
From Lemma \ref{lemma three} we have, for some $s_0> \beta^{-1}$,
\begin{align*}
\int_{B_3} v^\frac{p_0}{s_0} \om^n \int_{B_3}v^\frac{-p_0}{s_0   } \om^n
\leq C \left( \int_{B_3} v^{p_0} \om_0^n \int_{B_3}v^{-p_0   } \om_0^n \right)^\frac{1}{s_0}
\leq C \;.
\end{align*}
The above inequality gives the wanted inequality \eqref{reduced claim harnack} with
$p = \frac{p_0}{s_0}$.
The proof of the proposition is therefore achieved.
\end{proof}
As a result we have the following estimates.
\begin{prop}(The Harnack inequality)\label{linearharnack}
For any $B_{4d}(y) \subset \mathfrak X$,
suppose that $v$ is a non-negative $W^{1,2}$ solution of \eqref{linear equ ge} with homogeneous right hand side in a ball $B_{4d}(y) \subset \mathfrak X$.
Then, there is a constant $C$ depending on $(|b^i|_g+|c|_\infty)d$, $\b$  such that
\begin{align}\label{eqharnack}
\sup_{B_d} v\leq C\inf_{B_d} v.
\end{align}
\end{prop}
\begin{prop}(Interior H\"older estimate)\label{itholder}
Suppose that $v$ is a $W^{1,2}$ solution of \eqref{linear equ ge} in $\mathfrak X$ and suppose that
$f\in L^{\frac{q}{2}}$ and $h^i\in L^{q}$ with $q >m$.
Then, for any $B_{d_0}(y)\subset \Int\mathfrak X$ and $d\leq d_0$, there is a constant $C(|b^i|_g,|c|_\infty,d_0,q)$ and $\a((|b^i|_g+|c|_\infty)d_0,q)$ such that
\begin{align*}
osc_{B_d(y)} v\leq C d^{\a}(d_0^{-\a}\sup_{B_{d_0}(y)}|v|+d^{2(1-\frac{m}{q})}||f||_{\frac{q}{2}} +d^{1-\frac{m}{q}}\|h^i\|_q).
\end{align*}
\end{prop}

\subsubsection{Local estimates at the boundary}
 Consider a point $y \in \partial \mathfrak{X}$ and using the local holomorphic coordinate in the half space $R_+^{2n+2}=\{x\vert x^{n+1}\geq 0\}$, here $x^{n+1}$ is the real part of the variable $z^{n+1}$. Then the coordinate chart near $y$ becomes a domain $T$ in $R_+^{2n+2}$. Recall that we assumed $v_0$ in
$C_\beta^\alpha (\partial \mathfrak X )$ in \eqref{linear equ ge co}.
We let
$$
M : = \sup_{ \partial\mathfrak{X}} \cap B_{2d} v
\,, \qquad m : = \inf_{ \partial\mathfrak{X}} \cap B_{2d} v \; .
$$
Moreover we extend $v$ from the half space to the whole space $R^{2n+2}$.
\begin{equation*}
v_M^+ :=
\left\{
\begin{array}{cc}
\sup\{ v(x) , M \}, & x \in T\\\
M & x\not\in T\; .
\end{array}
\right.
\end{equation*}
\begin{equation*}
v_m^- :=
\left\{
\begin{array}{cc}
\inf\{ v(x) , m \}, & x \in T\\\
m & x\not\in T\;.
\end{array}
\right.
\end{equation*}
Just by following the proof of interior estimates, we obtain
\begin{prop}(Local boundedness at the boundary)
Suppose that  $v$ is a $W^{1,2}$ sub-solution of \eqref{linear equ ge} and suppose that
$f \in L^{q}, $ and $h^i\in L^{q}$, $i=1,\cdots, n+1$ with $ q > m$.
Then for any ball $B_{2d}(y)$ and any $p>1$
there is a constant $C$ depending on $(|b^i|_g+|c|_\infty)d$, $q$, $\b$, $p$ such that
$$\sup_{B_{d} (y)} v_M^+ \leq C( d^{-\frac{m}{p}} \| v_M^+ \|_{L^p (B_{2d} (y))}+
d^{2(1-\frac{m}{q})}||f||_{\frac{q}{2}} +d^{1-\frac{m}{q}}\|h^i\|_q).
$$
\end{prop}
\begin{prop}(Weak Harnack inequality at the boundary)
Suppose that $v$ is a  $W^{1,2}$ super-solution of \eqref{linear equ ge},
non-negative in a ball $B_{4d}(y) \cap T$ and suppose that  $f\in L^{\frac{q}{2}}$ and $h^i\in L^{q}$, $i=1,\cdots, n+1$ with $q >m$.
Then, for any $\frac{n+1}{n}>p>1$ there is a constant $C$ depending on $(|b^i|_g+|c|_\infty)d$, $q$, $\b$, $p$ such that
\begin{align*}
d^{-\frac{m}{p}} \| v_m^- \|_{L^p (B_{2d} (y))}
\leq
C(\inf_{B_{d(y)}} v_m^- +
d^{2(1-\frac{m}{q})}||f||_{\frac{q}{2}} +d^{1-\frac{m}{q}}\|h^i\|_q).
\end{align*}
\end{prop}
\begin{prop}(H\"oder estimate at the boundary)\label{eqholder}
Suppose that $v$ is a $W^{1,2}$ solution of \eqref{linear equ ge} in $\mathfrak X$ and
$f\in L^{\frac{q}{2}}$ and $h^i\in L^{q}$, $i=1,\cdots, n+1$ with $q >m$. Suppose that $y$ is on the boundary of $\mathfrak X$.
Then, for any $B_{d_0}(y)$ and $d\leq d_0$, there is a constant $C(|b^i|_g,|c|_\infty,d_0,q)$ and $\a((|b^i|_g+|c|_\infty)d_0,q)$ such that
\begin{align*}
\osc_{B_d(y)\cap\mathfrak X } v\leq C \{d^{\a}(d_0^{-\a}\sup_{B_{d_0}(y)\cap T}|v|+d^{2(1-\frac{m}{q})}||f||_{\frac{q}{2}} +d^{1-\frac{m}{q}}\|h^i\|_q)
+\osc_{B_{\sqrt{d_0d}}(y)\cap\p\mathfrak X } v\}.
\end{align*}
\end{prop}

\subsection{The Dirichlet problem of the linearized problem}
\label{The Dirichlet problem of the linearized problem}
Fix $0< \b <1$, and write $\mu:=\b^{-1}-1$.
Denote by $G$ the Green function of the standard cone metric $dr^2+\b^2r^2d\theta^2$ and by $T$ one of the second order operators
$$
\frac{\p^2}{\p s^i\p s^j},r^{-1}\frac{\p^2}{\p\theta\p s^i},\frac{\p^2}{\p r\p s^i}\;.
$$
It is shown in Donaldson \cite{MR2975584} Proposition 4 that the polyhomogeneous expansion of the Green function
around the singular set $D$
\begin{align}\label{green function}
G=\sum_{j,k}a_{j,k}(s)r^{\nu+2j}\cos k(\theta-\theta')\;.
\end{align}
Donaldson proved the following Schauder estimate.
\begin{prop}(Donaldson \cite{MR2975584})\label{dol2a}
Suppose that $\a \in (0,\mu)$,
then there exists a constant $C$ which depends only on $\a , \, m , \, \b $
such that for all functions $\rho \in C_c^\infty (\mathbb{R}^m)$, we have
\begin{align*}
 [i\partial \bar \partial (G\rho)]_\a \leq C[\rho]_\a \;.
\end{align*}
\end{prop}
\begin{rem}\label{pure}
The Schauder estimates of the remainder pure second order derivatives (i.e. $\p\p$ direction) of $G\rho$ are proved in Brendle \cite{Brendle:2011kx} (Proposition A.1) and Jeffres, Mazzeo and Rubinstein \cite{Jeffres:2011vn} (Proposition 3.3).
\end{rem}
In our problem, the interior Schauder estimate follows
from Proposition \ref{dol2a}.
When we consider the Schauder estimate near the boundary,
we notice that our manifold is a product manifold, then the new Green function is constructed by replacing $s$ by $(s,x^{n+1},y^{n+1})$ in \eqref{green function}.
So when $i$ or $j$ is not equal to $n+1$,
then the $\p_i\bar \p_j$ estimate follows exactly the same line of Donaldson's proof (and further regularity by Jeffres, Mazzeo and Rubinstein \cite{Jeffres:2011vn}).
The $\p_{n+1}\bar \p_{n+1}$ estimate follows from the equation (cf. Section 4.4 in \cite{MR1814364}).
Now we patch the local estimates to the whole manifold by the partition of unity in the standard way.
\begin{prop}\label{linear equ gl}
Fix $\a$ with $0<\a<\mu=\b^{-1}-1$.
Then there is a constant $C$ depending on $\b$, $m$, $\a$
such that for all the functions $f\in C^\a_\b$ we have the Schauder estimate of the weak solution
of the equation \eqref{linear equ ge}
$$
|v|_{C^{2,\a}_\b}\leq C(|v|_{L^{\infty}}+|f|_{C^\a_\b}+\sum_i|h^i|_{C^{1,\a}_\b})\;.
$$
\end{prop}
Combining the existence and uniqueness of the weak solution Proposition \ref{ext weak}, we obtain
\begin{prop}\label{dp}
There exists a unique solution of \eqref{linear equ ge} with data as \eqref{linear equ ge co} in $C^{2,\a}_\b$.
\end{prop}
The linear theory in this section immediately implies the $\p\bar\p$-lemma with cone singularities.
\section{The metric space structure}\label{metricstr}
In this section we apply our geodesic
to study the geometry of the space of K\"ahler cone metrics.
We equip the space of K\"ahler cone metrics with the following normalization condition; we ask any K\"ahler cone potential $\vphi$ with respect to the background modes metric $\om$to satisfies $I(\vphi)=0$ vanishes, where
\begin{align*}
I_\om(\vphi)
&=\frac{1}{V}\int_M\vphi\om^n-\frac{1}{V}\sum_{i=0}^{n-1}\frac{i+1}{n+1}
\int_{M}\p\vphi\wedge\bar\p\vphi\wedge\om^{i}\wedge\om_\vphi^{n-i-1}.
\end{align*}
In particular, the functional $I(\vphi)$ is well defined along the $C^{1,1}_\b$ geodesic.
We show that the space of cone metrics
has a structure of metric space following the approach in \cite{MR1863016}. We said that $\vphi (t)$ is an $\eps$-approximate geodesic if it solves
\begin{align} \label{eqn: approximate geodesic }
(\vphi '' - |\p \vphi '|^2_{g_\vphi}) \det g_\vphi = \eps f \det g,
\end{align}
where $f = \frac{\det \Om_1}{\det \Om}
= | m \Phi_{n+1 , \overline{n+1}}
-\p(\vphi(1) - \vphi(0))|_\Om$.
Recall the energy $E:= \int_0^1 \int_M \vphi ' (t) \om_{\vphi(t)}^n dt $.
Along the $C^{1,1}_\b$ geodesic, there holds
\begin{align}\label{eqn: estimate of the energy of positive length of geodesic arc}
 \frac{1}{2} \left| \frac{d}{dt} E \right|
= \left| \int_M \vphi ' (\vphi '' - |\p \vphi|_{g_\vphi}^2 ) \om_\vphi^n \right|
\leq \eps \sup_{\mathfrak{X}}|\phi '| \cdot \sup_{\mathfrak{X}} |f| \cdot \Vol \; .
\end{align}

We show positivity of the length of any non-trivial geodesic segment and the geodesic approximation lemma. We omit the proof here, since along the $C^{1,1}_\b$ geodesic, all the inequalities are well defined.
\begin{prop}\label{prop: geodesic length is positive}
Let $\vphi (t) $ be a  $ C^{1, 1}_\b$ geodesic from $0$ to $\vphi$, and $I(\vphi)=0$.
Then the following inequality holds
\begin{align*}
 \int_0^1 \sqrt{\int_M (\vphi ')^2 \frac{\om_\vphi^n}{n!}}dt
\geq {\Vol}^{-\frac{1}{2}}\left(
\sup \left(
\int_{\vphi>0} \vphi \frac{\om_\vphi^n}{n!}\, , \, \int_{\vphi<0} \vphi \frac{\om_0^n}{n!}
\right)
\right)\; .
\end{align*}
In particular, the length of any non-constant $C^{1,1}_\b$ geodesic is positive.
\end{prop}
\begin{lem}\label{approximate}
 Let $\mathcal{H}_C \subset \mathcal{H}_\b$ be as in Definition \ref{spaceHC}.
Also, let $C_i:= \vphi_i (s) : [0,1] \to \mathcal{H}_C$, for $i=1,2$, be two smooth curves.
Then, for a small enough $\eps_0$, there is a
two-parameter family of curves
\begin{align*}
 C( s , \eps) : \vphi(t , s, \eps) : [0,1]\times[0,1]\times (0, \eps_0] \to \mathcal{H}
\end{align*}
such that the following properties hold:
\begin{enumerate}
 \item Fixed $s,\eps$, then $C(s,\eps)\in C_\b^{2,\a}$ is an $\eps$-approximate geodesic from
$\vphi_1 (s)$ to $\vphi_2(s)$.
\item There exists a uniform constant $C$ such that
\begin{align*}
 |\vphi|+ \left| \frac{\p \vphi}{\p t}\right| + \left| \frac{\p \vphi}{\p s}\right|< C; \quad
0\leq \frac{\p^2 \vphi}{\p t^2} < C ; \quad \frac{\p^2 \vphi }{\p s^2} < C \; .
\end{align*}
\item Fixed any $s$, the limit in $C_\b^{1,1}$ of $C(s,\eps)$ as $\eps\to 0$ is the unique
geodesic arc from $\vphi_1(s)$ to $\vphi_2(s)$.
\item There exists an uniform constant $C$ such that, about the energy $E(t,s,\eps)$
along the curve $C(s,\eps)$, there holds
\begin{align*}
 \sup_{t,s}\left| \frac{\p E }{\p t}\right| \leq \eps \cdot C \cdot \Vol \; .
\end{align*}
\end{enumerate}
\end{lem}
With the geodesic approximation lemma above, the triangular inequality and the differentiability property of the distance function follow immediately.
\begin{thm}
 Suppose that $\phi= \vphi (s) : [0,1]\to \mathcal{H}_\b$ is a smooth curve, and
let $p$ be a base point of $\mathcal{H}$. Then, the length of the geodesic arc
between $p$ and $\vphi$ is less than the sum of the length of the geodesic arc
between from $p$ to $\phi(0)$ and the length of the curve from $\phi(0)$ to $\phi(s)$.
\end{thm}
\begin{thm}
The distance function given by the length of the geodesic arc is a differentiable function.
\end{thm}

\
\bibliography{bib}
\bibliographystyle{plain}
\end{document}